\documentclass{article}

\title{An effective model for boundary vortices in thin-film micromagnetics 
}
\author{{\Large Radu Ignat}\footnote{Institut de Math\'ematiques de Toulouse \& Institut Universitaire de France, UMR 5219, Universit\'e de Toulouse, CNRS, UPS
IMT, F-31062 Toulouse Cedex 9, France. Email: Radu.Ignat@math.univ-toulouse.fr} \and {\Large Matthias Kurzke}\footnote{School of Mathematical Sciences,
University of Nottingham, University Park, Nottingham, NG7 2RD, UK. Email: matthias.kurzke@nottingham.ac.uk}}

\usepackage[mathscr]{euscript}
\usepackage{tikz} 
\usetikzlibrary{decorations.pathreplacing,positioning, arrows.meta}
\usepackage{pstricks} 
\usepackage{graphicx, pst-plot, pst-math} 
\usepackage{amsfonts, dsfont}
\usepackage{amsmath, latexsym}
\usepackage{amssymb,verbatim}
\usepackage{ntheorem}
\usepackage[normalem]{ulem}
\usepackage[disable]{todonotes}
\usepackage{enumerate}
\usepackage{color}

\newtheorem{lem}  {Lemma}
\newtheorem{pro}[lem]    {Proposition}
\newtheorem{thm}[lem]   {Theorem}
\newtheorem{cor}[lem]    {Corollary}
\newtheorem{df}[lem]     {Definition}

\theorembodyfont{\rmfamily}
\newtheorem{rem}[lem]{Remark}

\numberwithin{equation}{section}

\def\XXint#1#2#3{{\setbox0=\hbox{$#1{#2#3}{\int}$}
     \vcenter{\hbox{$#2#3$}}\kern-.5\wd0}}

\newcommand{\FF}{{\mathcal F}}
\newcommand {\den} {{e}_\eta}
\newcommand {\ka}{\varkappa}

\newcommand{\nltL}[2]{\|#1\|_{L^2({#2})}}
\newcommand{\jacbd}{\mathcal{J}_{bd}}
\newcommand{\jaco}{\mathcal{J}}
\newcommand{\Ss}{\mathbb{S}}

\newcommand{\RR}{\mathbb{R}}
\newcommand{\R}{\mathbb{R}}
\newcommand{\CC}{\mathbb{C}}

\newcommand{\eps}{\varepsilon}

\newcommand{\h}{{\mathcal{H}}}

\newcommand{\proof}[1]{\par\medskip\noindent{\bf Proof#1.}}
\newcommand{\qed}{\hfill$\square$}

\newcommand{\be}{\begin{equation}}
\newcommand{\ee}{\end{equation}}

\newcommand{\nd}{\noindent}

\newcommand{\NN}{\mathbb{N}}
\newcommand{\ZZ}{\mathbb{Z}}
\newcommand{\Om}{\Omega}
\newcommand{\dOm}{\partial\Omega}
\newcommand{\dist}{\mathop{\rm dist \,}}
\newcommand{\diam}{\mathop{\rm diam \,}}
\newcommand{\jac}{\mathop{\rm jac \,}}

\newcommand{\f}{\varphi}
\newcommand{\rh}{\rho_h}

\renewcommand{\O}{\Omega}

\newcommand{\degr}{\operatorname{deg}}

\newcommand{\bm}{\bar {\bf m}_h}
\newcommand{\mk}{\bar m_{h,k}}
\newcommand{\nah}{\nabla\cdot \bar m_h}

\newcommand{\bmt}{\bar m_{h,3}}
\newcommand{\bmp}{\bar m_{h}}
\newcommand{\bmn}{(\bar m_h\cdot \nu)}
\newcommand{\bu}{\bar U_h}
\newcommand{\eeeh}{\bar E_h}
\newcommand{\eee}{E_{\varepsilon,\eta}}

\newcommand{\mcS}{\mathcal{S}}

\newcommand{\bfm}{\mathbf{m}}
\newcommand{\bfM}{\mathbf{M}}
\newcommand{\bfa}{\pmb{\nabla}}
\newcommand{\bfd}{\pmb{\Delta}}
\newcommand{\bfx}{\mathbf{x}}
\newcommand{\bfy}{\mathbf{y}}
\newcommand{\bfxi}{\pmb{\xi}}
\newcommand{\bfom}{\pmb{\omega}}
\newcommand{\Bfom}{\pmb{\Omega}}

\newcommand{\bfnu}{\pmb{\nu}}

\newcommand{\ci}{\mathrm{i}}

\newcommand{\mmregime}[7]{
  \pgfmathsetmacro{\first}{(#2)*10 + #3 - .9} 
    \pgfmathsetmacro{\last}{(#4)*10 + #5 - 1.1} 
  \pgfmathsetmacro{\middle}{(\first+\last)/2} 
  \fill[#7] (\first,#6-0.95) rectangle (\last,#6-0.05) (\middle,#6-.5) node[white, font=\sf\small]{#1};
}

\begin{document}

\maketitle

\listoftodos

\begin{abstract}
Ferromagnetic materials are governed by a variational principle
which is nonlocal, nonconvex and multiscale. The main object is given by a unit-length three-dimensional vector field, the magnetization, that corresponds to the stable states of the micromagnetic energy. Our aim is to analyze a thin film regime that captures the asymptotic behavior of boundary vortices generated by the magnetization  
and their interaction energy. This study is based on the notion of ``global Jacobian" 
detecting the topological defects that \emph{a priori} could be located in the interior and at the boundary of the film.
A major difficulty consists in estimating the nonlocal part of the micromagnetic energy 
 in order to isolate the exact terms corresponding to the topological defects. We prove the concentration of the energy around boundary vortices via a $\Gamma$-convergence expansion at the second order. The second order term  is  the renormalized energy that represents the interaction between the boundary vortices and governs their optimal position. We compute the  expression of the renormalized energy 
 for which we prove the existence of minimizers having two boundary vortices of multiplicity $1$.
 Compactness results are also shown for the magnetization and the corresponding global Jacobian. 


\end{abstract}

\tableofcontents



\section{Introduction} 

The aim of the paper is to develop a mathematical analysis of thin ferromagnetic films in a \emph{boundary vortex} regime. This regime is characterized
by  the concentration of the  micromagnetic energy around topological point singularities located at the boundary of the film, the so-called boundary vortices. Our mathematical approach consists in  determining an  asymptotic expansion by $\Gamma$-convergence 
that is precise enough to  capture the interaction energy between the boundary vortices that governs their location at the boundary of the sample.

\medskip

We start by a brief introduction to micromagnetics which is a nonconvex, nonlocal and multiscale variational principle. For that, we consider 
a ferromagnetic sample of cylindrical shape
 $$\bfom = \omega \times (0,t)\subset \R^3$$ where $\omega \subset \RR^2$ is the transversal section with  diameter $\ell>0$ 
   and $t>0$ is the thickness of the cylinder $\bfom$. We assume that $\omega \subset \RR^2$ is a simply-connected domain of class $C^{1,1}$.
   The behavior of the magnetic moments inside the sample is described by the magnetization $\bfm$ which is a three-dimensional ($3D$) unit-length vector field $$\bfm=(m, m_3):\bfom \to \Ss^2, \quad m=(m_1,m_2).$$ Here 
{\bf bold symbols} always denote a $3D$ quantity, {\it italic symbols} denote $1D$ or $2D$ quantities and $\Ss^2$ is the unit sphere in $\R^3$. The 
{magnetization represents a stable state} of 
the micromagnetic energy, considered here in the absence of anisotropy and applied magnetic field:
\begin{align*}  %
 E^{3D}(\bfm) = A^2 \ \int_{\bfom} |\bfa \bfm|^2 \ d\bfx 
  + \int_{\RR^3} |\bfa U|^2 \ d\bfx,
 \end{align*}
where $\bfx=(x, x_3)\in \R^3$ stands for the $3D$ space variable with the differential
operator
$\bfa=(\nabla, \partial_{x_3})$, while $x=(x_1, x_2)\in \R^2$ and $\nabla=(\partial_{x_1}, \partial_{x_2})$. The first term in $E^{3D}$ is called 
\emph{exchange energy} and penalizes the variations of $\bfm$ according to the exchange length of the material $A>0$ that is 
typically on 
 the order of nanometers. 

The second term is the nonlocal 
\emph{magnetostatic energy} that is carried by
the $H^1$ stray-field potential
 $U:\RR^3\to \RR$ generated by $\bfm$ via the static Maxwell's equation:
 $$
\bfd U = \bfa\cdot (\bfm \mathds{1}_{\bfom}) \quad \textrm{ in the sense of distributions in } \RR^3,
$$ where $\mathds{1}_{\bfom}$ is the characteristic function of $\bfom$, 
\be
\label{Helmholtz}
\textrm{i.e., }\quad \int_{\RR^3} \bfa U \cdot \bfa \zeta \ d\bfx = \int_{\bfom} \bfm \cdot \bfa \zeta \ d\bfx , \quad \textrm{ for every }
    \zeta\in C^\infty_c(\RR^3) \ee
    (see Appendix \ref{appendixA}). In other words, the stray field $\bfa U$ is the $L^2$ Helmholtz projection onto gradient fields of the magnetization $\bfm$ (extended by zero outside the sample $\bfom$).

The combination of nonconvexity (through the saturation condition $|\bfm|=1$ in $\bfom$) and nonlocality
(through the coupling with Maxwell's equation) leads to rich magnetic
pattern formation. Depending on the length scales of the system (e.g. the exchange length $A$, the thickness $t$ and width $\ell$ of the sample $\bfom$), experiments have shown different types of singular patterns, such as domain walls and vortices. To predict and 
describe  these  microscopic structures occurring across a wide range of spatial scales poses many
 challenging mathematical problems that have led to a substantial amount of 
recent research, see e.g. DeSimone et al. \cite{DeSimoneKohnMuller:2005a} and Ignat \cite{Ignat_HDR}. 

\subsection{A thin film regime}
 We are interested in analyzing the asymptotic behavior of the magnetization and of the micromagnetic energy $E^{3D}$ in a special thin film regime
 where boundary vortices appear. For that, it is convenient to use instead of the three length scales $\ell$, $t$ and $A$ of the model, 
 the following
  two dimensionless parameters\footnote{ $h$ stands for the aspect ratio of the sample $\bfom$, while $\eta$, the reduced exchange length, corresponds to the core size of interior vortices. Both parameters $h$ and $\eta$ will be small in our regime.}:
  \begin{align*}
    h:= \frac{t}{\ell}>0   && \text{and} && \eta := \frac A \ell>0.  
  \end{align*}
  
 \medskip 
 
 {\nd \bf Rescaling}.   
 We consider $\eta=\eta(h)$ as a function of $h$ and nondimensionalize in length with respect to $\ell$, i.e. $\hat \bfx=\frac{\bfx}{\ell}$, 
 $$ \hat \bfx\in\Bfom_h=\Omega\times (0,h)\subset \R^3, \quad \Omega=\frac{\omega}{\ell}\subset \R^2,$$ 
(so, the transversal section $\Omega$ has diameter one), $\bfm_h(\hat \bfx)=\bfm(\bfx)$, $U_h(\hat \bfx)= \frac 1 \ell U(\bfx)$. Then we 
rescale the energy at the energetic level of boundary vortices  to $\hat E_h(\bfm_h)=\frac{1}{A^2 t |\log \eps|} E^{3D}(\bfm)$, where  
 $$
 \eps=\frac{\eta^2}{h|\log h|}$$
 is a function $ \eps(h)$ of $h$ standing for the core size of a boundary vortex. We will always assume $h\ll 1$ and $\eps\ll 1$.
 {\bf Skipping the hat $\ \hat{} \ $ from now on}, we will use the following quantities:
 \be \label{eee}
    E_h(\bfm_h) = \frac{1}{h|\log \eps|} \ \int_{\Bfom_h} |\bfa \bfm_h|^2 \ d\bfx + \frac{1}{\eta^2 h |\log \eps|}\int_{\RR^3} | \bfa
   U_h|^2 \ d\bfx,
   \ee
where 
\be
\label{strayfield}
 \bfm_h:\Bfom_h\to \Ss^2, \quad \bfd U_h=\bfa \cdot (\bfm_h \mathds{1}_{\Bfom_h}) \quad \textrm{ in  } \quad \RR^3.
\ee


\nd {\bf Notation}. We use the following notation: $a\ll b$ or $a=o(b)$ if $\frac a b\to 0$, as well as $a\lesssim b$ or $a=O(b)$ 
 if there exists a universal $C>0$ such that $a\leq  Cb$ and finally, $a\sim b$ if $a\lesssim b$ and  $b\lesssim a$.
We adopt the constant convention, which means that constants called $C$ can change their value from one line to the next. 
The notation ``a sequence / family $h\to 0$'' or similar will be used both to refer to a sequence $h_k\to 0$ or to a continuous parameter $h\to 0$.

\medskip

\nd {\bf Regime}. We will focus on the following regime of thin films (i.e., the aspect ratio $h\ll 1$ is very small) where the {rescaled} energy $E_h$ in \eqref{eee} is expected to concentrate around boundary vortices: 
\be
\label{regime}
\eta, h \ll 1 \quad \textrm{and} \quad \frac{1}{|\log h|}\ll \eps \ll 1.\ee
Note that \eqref{regime} is equivalent with $h\ll {\eta^2}\ll h|\log h|\ll 1$, so that $|\log h|\sim |\log \eta|$. As $a\ll b \ll 1$ implies $a|\log a |\ll b |\log b | \ll 1$, we deduce for $a=\frac{1}{|\log h|}$ and $b=\eps$:
\begin{equation}\label{regimelogs}
\frac{\log|\log h|}{|\log h|} \ll \eps|\log\eps|.
\end{equation}
Also note that  in terms of $\eps$ and $\eta$, \eqref{regime} can be written as
$$
\eta \ll 1 \quad \textrm{and} \quad \frac{1}{|\log \eta|}\ll \eps \ll 1.$$
In particular, one has $|\log \eps|{\leq} \log |\log \eta|$ and 
as $|\log h|\sim |\log \eta|$, it implies 
{$|\log \eps|\lesssim \log|\log h|$}. 
Some of our results are only valid in the following regime, which is narrower than \eqref{regimelogs}:
\be
\label{regime2}
\frac{\log|\log h|}{|\log h|} \ll  { \eps}
\ee
Obviously, { if $\eta(h), \eps(h), h\ll 1$, then \eqref{regime2} implies \eqref{regime}.}
Possible choices for $\eta(h)$ for which $\eps(h)$ satisfies \eqref{regime2} are $\eta^2=Ch|\log h|^\beta$ for some $0<\beta<1$ and $C>0$. 
The choice $\eta^2=Ch\log|\log h|$ with $C>0$ is an example for which \eqref{regime} holds true but \eqref{regime2} does not.

Our regime \eqref{regime} fills the gap in the analysis of thin film regimes in micromagnetics between the 
regimes $\eta^2=O(h)$ studied by Moser \cite{Moser:2004a}
and $\eta^2=O(h|\log h|)$ analysed by Kohn-Slastikov \cite{KohnSlastiko:2005a}.
We refer to Subsection \ref{physi} for a discussion  
of the thin film regimes in micromagnetics.

\medskip 

\nd {\bf A reduced  $2D$ model}.
A key point in our analysis concerns the behavior of the micromagnetic energy $E_h$ (given in \eqref{eee}) in the asymptotic regime \eqref{regime}. As the aspect ratio $h$ tends to $0$, the appropriate quantity to study is the vertical average magnetization $\bfm_h$ (given in \eqref{strayfield}) defined in the $2D$ section $\Omega$:
\be\label{eq:averageofm}
\bm(x)=\frac{1}{h} \int_0^h \bfm_h(x, x_3)\, dx_3, \quad x\in \Om\subset \R^2,
\ee
to which the stray field potential $\bu:\R^3\to \R$ is associated via 
\be
\label{stray_mean}
\bfd \bu=\bfa \cdot (\bm {\mathds 1}_{\Bfom_h}) \quad \textrm{ in }\quad \RR^3. 
\ee
Note that the unit-length constraint is convexified by averaging, i.e. 
$|\bfm_h|=1$ yields $|\bm|\le 1$, so
$$\bm=(\bmp, \bmt):\Om\to \bar B^3$$ 
where $\bar B^3$ is the closed unit ball in $\R^3$. 

We will prove that in the regime \eqref{regime},  the $3D$ micromagnetic model reduces to a $2D$ model for the average magnetization $\bm$. The major difficulty consists in determining the scaling of the nonlocal part of the rescaled energy $E_h$ in the regime \eqref{regime}. More precisely, the stray field energy penalizes the distance of the in-plane average $\bmp$ to the unit circle $\Ss^1$ inside $\Omega$ as well as the normal component $\bmn$ at the boundary $\partial \Om$ where $$\nu=(\nu_1, \nu_2)$$ stands for the outer unit normal vector at  $\partial \Om$. 
Thus, we introduce the following reduced $2D$ energy functional associated to the average magnetization $\bm$ with $|\bm|\le 1$ defined in the $2D$ simply-connected $C^{1,1}$ domain $\Omega$:
\begin{equation}\label{eq:defofebar}
\eeeh(\bm)= \frac{1}{|\log \eps|} \bigg( \ \int_{\Omega} |\nabla \bm|^2 \ dx+ \frac{1}{\eta^2}\int_{\Om} (1-|\bmp|^2) \ dx+\frac{1}{2\pi \eps} \int_{\partial \Om} \bmn^2\, d\h^1\bigg).
\end{equation}
We can extend the definition of $\eeeh$ to all of $H^1(\Om;\R^3)$ by setting $\eeeh(\mathbf{m})=\infty$ if $|\mathbf{m}|>1$ on a set of positive measure.

\begin{thm}\label{thm:lowerb} Let $\Bfom_h=\Omega\times (0,h)$ with $\Om\subset \R^2$ be a simply connected 
$C^{1,1}$ domain. In the regime \eqref{regime}, we consider a family of magnetizations $\{\bfm_h:\Bfom_h\to \Ss^2\}_{h\to 0}$ 
with associated  stray field potentials $\{U_h:\R^3\to \R\}_{h\to 0}$ given by \eqref{strayfield} and we assume  $$ \limsup_{h\to 0} E_h(\bfm_h)<\infty.$$ 
Then
$$E_h(\bfm_h)\ge \eeeh(\bm)-o(1) \quad \textrm{as} \quad h\to 0.$$
Moreover, in the more restrictive regime 
\eqref{regime2},
we have the following improved estimate:
\be\label{eq:improvedest}
E_h(\bfm_h)\ge \eeeh(\bm)-o(\frac{1}{|\log\eps|})\quad \textrm{as} \quad h\to 0.\ee
If $\bfm_h$ are independent of $x_3$ (i.e., $\bfm_h=\bm$), then 
in the regime \eqref{regime} there holds 
$E_h(\bfm_h)= \eeeh(\bm)-o(1)$, while in the regime \eqref{regime2} we have $E_h(\bfm_h)= \eeeh(\bm)-o(\frac{1}{|\log\eps|})$
as $h\to 0$.
\end{thm}

 Let us highlight the role of the above estimates: while the full micromagnetic energy $E_h$ is nonlocal (due to the stray field), the reduced energy $\eeeh$ becomes local in terms of the average magnetization, so easier to handle.\footnote{Physically, the  regime \eqref{regime} corresponds to fairly small magnetic samples where the nonlocality is lost, e.g., N\'eel walls cannot nucleate since their core is too wide to be contained in such small samples.
 However, the samples are still large compared to the core size of a boundary vortex. See Section \ref{physi} for more details.} Moreover, the improved energy estimate \eqref{eq:improvedest} 
is essential to carry out the  asymptotic $\Gamma$-development of the $3D$ energy $E_h$ at the
 second order that allows us to determine the interaction between boundary vortices.

In \cite{IK_jac}, we have studied the energy functional for $2D$ maps $u\in H^1(\Omega;\R^2)$:
\[
E_{\eps,\eta}(u) = \int_\Omega |\nabla u|^2 dx +  \frac1{\eta^2} \int_{\Omega} (1-|u|^2)^2dx + \frac1{2\pi \eps} \int_{\partial\Omega} (u\cdot\nu)^2 d\h^1.
\]
Comparing this to $\eeeh$, we have for $u=\bmp$ that $(1-|u|^2)\ge (1-|u|^2)^2$. Since $|\nabla (\bm\cdot e_3)|^2\ge 0$, we deduce 
\begin{equation}\label{eundikjace}
\eeeh(\bm) \ge \frac1{|\log\eps|} E_{\eps,\eta}(\bmp),
\end{equation}
and so the lower bounds obtained for $E_{\eps,\eta}$ in \cite{IK_jac} will provide lower bounds for $\eeeh$ and in turn for $E_h$ (by Theorem \ref{thm:lowerb}). For maps $\bm=(\bmp,0)$ with values into $\Ss^1\times \{0\}$,  
\eqref{eundikjace} is an equality and will allow us to obtain matching upper bounds. See Section~\ref{sec:sec3} for details.

\subsection{Global Jacobian }
We want to explain now the topological challenges carried by the reduced $2D$ energy $\eeeh$ defined in \eqref{eq:defofebar}. It is similar to the standard Ginzburg-Landau functional studied in the seminal book of Bethuel-Brezis-Helein \cite{BethuelBrezisHelein:1994a}, i.e.,
\be
\label{ginlan}
u\in H^1(\Omega\subset \R^2; \R^2)\mapsto \int_{\Omega} \den(u)\, dx \quad \textrm{with } \, \den(u)=|\nabla u|^2+\frac{1}{\eta^2} (1-|u|^2)^2,\quad \eta>0.\ee
However, the reduced energy $\eeeh$ leads to richer singular pattern formation
due to the additional penalty term at the boundary. Indeed, we expect that stable states of $\eeeh$ generate both interior and boundary vortices (see Moser \cite{Moser:2003a}).\footnote{In \cite{Moser:2003a}, Moser studies \eqref{eee} neglecting the out-of-plane component of the magnetization and in the special regime where $\eps=\eta^\alpha$, with $\alpha\in (0, 1]$. The author proves that global minimizers nucleate  two boundary vortices if $\alpha<1$,  while for $\alpha=1$, either two boundary vortices or an interior 
vortex are possible. For a similar problem, a more refined analysis was performed by Alama et al. \cite{AlamaBronsardGalvao2015}. }
 Therefore, we need to define a notion of {\it global Jacobian} capable of detecting topological singularities in the interior as well as at the boundary of the sample. We refer to our previous paper \cite{IK_jac} for more details. In our setting, the $2D$ map $u$ in \eqref{ginlan} plays the role of the in-plane components of the average magnetization $\bm$, i.e., $u=(\bm\cdot {\bf e}_1, \bm\cdot {\bf e}_2)$ where $({\bf e}_1, {\bf e}_2, {\bf e}_3)$ is the Cartesian basis in $\R^3$.  

\medskip

\nd {\bf Global Jacobian}. 
For a $2D$ map $u\in H^1(\Omega;\R^2)$  defined on the $C^{1,1}$ 
domain $\Omega\subset \R^2$, we call {\it global Jacobian} of $u$ the following linear operator $\jaco(u):W^{1, \infty}(\Omega)\to \RR$ acting on Lipschitz test functions:\footnote{\label{fn3}
Note that $u \times \nabla u\in L^1(\Om;\R^2)$ for $u\in H^1(\Om;\R^2)$. More generally, the global Jacobian $\jaco(u)$ extends naturally to a map $u\in L^p(\Omega;\R^2)$ with $\nabla u \in L^q(\Omega;\R^{2\times 2})$ where $\frac1p+\frac1q=1$, $p, q\in [1, \infty]$; { in particular, this is the case for $u\in W^{1,1}(\Omega; \Ss^1)$.}} 
$$\left<\jaco(u), \zeta\right>:=-\int_{\Omega} u \times \nabla u\cdot \nabla^\perp \zeta\, dx, \quad \textrm{for every Lipschitz function }\, \zeta: \Omega\to \R,$$
where $\nabla^\perp=(-\partial_{x_2}, \partial_{x_1})$ and $a\times b=a_1b_2-a_2b_1$ for any two vectors $a=(a_1, a_2)\in \R^2$ and $b=(b_1, b_2)\in \R^2$. In particular, $$\left<\jaco(u), 1\right>=0.$$ 
On the one hand, the global Jacobian carries the topological information in the interior of $\Omega$, where
 it coincides (up to a multiplicative constant) with the standard Jacobian of $u$ detecting the interior vortices that we call {\it interior Jacobian} in the sequel, i.e., 
$$\jac(u)=\det (\nabla u)=\partial_{x_1}u\times \partial_{x_2}u\in L^1(\Omega).$$
Indeed, integration by parts for Lipschitz test functions $\zeta\in W^{1,\infty}_0(\Omega)$ vanishing at the boundary $\partial \Omega$ yields 
\[
\left<\jaco(u), \zeta\right> = 2 \int_{\Omega} 
\jac(u) \zeta \, dx \quad  \textrm{ if } \, \zeta=0 \,  \textrm{ on } \, \partial \Omega
\]
(see e.g. the proof of \cite[Proposition 2.1]{IK_jac}).
On the other hand, the global Jacobian also carries the topological information at the boundary $\partial \Omega$ and will enable us to detect the boundary vortices; 
more precisely, we define the {\it boundary Jacobian} of $u$ to be the linear operator
 $\jacbd(u):W^{1, \infty}(\Omega)\to \RR$ given by
 \be
 \label{def_jac_bd}
 \jacbd(u):=\jaco(u)-2\jac(u).\ee
 In fact, the operator $\jacbd(u)$ acts only on the boundary of $\partial \Om$: For  maps $u$ 
 that are smooth up to the boundary, integration by parts yields\footnote{A Lipschitz function $\zeta:\Omega\to \R$ has a unique Lipschitz extension to $\bar \Omega=\Omega\cup \partial \Omega$ that we consider (tacitly) in the following. } 
 $$\left< \jacbd(u), \zeta\right>=-\int_{\partial \Om} (u\times \partial_{\tau} u ) \zeta\,d\h^1 \quad \textrm{for every Lipschitz function $\zeta:\Omega\to \R$}$$
 (for the general case $u\in  H^1(\Omega;\R^2)$, see \cite[Proposition 2.2]{IK_jac}). 
 Here and throughout the paper, $$\tau=\nu^\perp=(-\nu_2, \nu_1)$$ is the unit tangent vector at $\partial \Omega$ 
 such that $(\nu, \tau)$ form an oriented frame,
 and we write $\partial_\tau$ for the derivative along the boundary.
 
Note that for unit-length maps $u:\Omega\to \Ss^1$ { that admit a smooth lifting $\f:\Omega\to \R$, i.e., $u=(\cos \f, \sin \f)$ in $\Omega$, }the interior Jacobian of $u$ vanishes, so that the whole topological information is carried by the
tangential derivative of the lifting $\f$ at the boundary:
 \be
 \label{jac_s1}
\jac(u)=0, \quad \jaco(u)=\jacbd(u)=-\partial_{\tau} \f\, {\h^1}\llcorner \partial \Omega \quad \textrm{and}\quad \left<\jacbd(u), 1\right>=0 \quad \textrm{ if } \, |u|=1 \,  \textrm{ in } \, \Omega.
\ee
Typically, in our model, the limiting global Jacobian in the regime \eqref{regime} is a measure supported on the boundary $\dOm$, of the form $J=-\ka \, {\h^1\llcorner \dOm}+\pi \sum_{j=1}^N d_j \delta_{a_j}$ 
for $N$ distinct boundary vortices $a_j \in \partial\O$ with  $d_j\in \ZZ\setminus\{0\}$ and $\ka$ the curvature on $\dOm$. 
The necessary condition $\left<J,1\right>=0$ yields, via the Gauss-Bonnet formula, the following topological constraint on the multiplicities $d_j$: 
$$\pi \sum_{j=1}^N d_j=\int_{\dOm}\ka\, d\h^1=2\pi, \quad \textrm{i.e.,} \quad \sum_{j=1}^N d_j=2$$ 
(see e.g. \cite{IK_jac}).

\subsection{$\Gamma$-expansion for boundary vortices}
{In Theorem~\ref{thm:lowerb} we have connected the $3D$ micromagnetic energy $E_h$ in \eqref{eee} to a reduced $2D$ energy for the average magnetizations that are defined on the $2D$ transversal section $\Omega$. This reduction result plays a fundamental role in proving our main Theorem \ref{thm:t1} below. First, we show that the global Jacobian of the average magnetizations converges (up to extraction) to a measure supported at the boundary $\partial \Omega$
 provided a certain energy bound  of the magnetizations in the regime \eqref{regime}. This limit measure involves a finite sum of Dirac measures at the boundary vortices $a_j\in \dOm$ having nonzero multiplicities $d_j\in \ZZ\setminus \{0\}$ that satisfy the topological constraint $\sum d_j=2$. Next, we prove a $\Gamma$-convergence result (at the first order) for the $3D$ energy  $E_h$ where the $\Gamma$-limit depends on the number of boundary vortices detected by the global Jacobian in the regime \eqref{regime}. Finally, if we restrict to the narrower regime \eqref{regime2}, we will prove a $\Gamma$-expansion at the second order of the $3D$ energy  $E_h$ that enables us to capture the positions of boundary vortices.
More precisely, the second order term in this $\Gamma$-expansion involves the renormalized energy, similar to that of Bethuel-Brezis-H\'elein \cite{BethuelBrezisHelein:1994a} (see \cite{IK_jac}) that is introduced via $\Ss^1$-valued canonical harmonic maps with prescribed boundary vortices.}

\bigskip

\nd {\bf Canonical harmonic maps and renormalized energy}. 
The canonical harmonic maps we consider in this paper are $\Ss^1$-valued smooth harmonic maps $m_*=e^{\ci \phi_*}$ in $\Om$ (i.e., $\Delta \phi_*=0$ in $\Om$) that are tangent on the boundary $\dOm$ except at $N$ boundary vortices $a_j\in \dOm$ where $m_*$ winds according to the multiplicities $d_j$. 

\begin{df}\label{defi:can_map}
Let $\Omega\subset \R^2$ be a bounded, simply connected, $C^{1,1}$ regular domain and $\nu$ be the unit outer normal field on $\dOm$ with the tangent field $\tau=\nu^\perp$ and the curvature $\ka$. For $N\geq 1$, we consider $N$ distinct points $a_1, \dots, a_N \in \dOm$ with the multiplicities $d_1, \dots, d_N \in \ZZ\setminus \{0\}$ such that $\sum_{j=1}^N d_j=2$ and a $BV$ function $\phi:\partial\Omega\to \R$ such that  $e^{\ci\phi}\cdot \nu=0$ on $\partial\Omega\setminus\{a_1,\dots,a_N\}$ and
 $$\partial_\tau \phi =\ka  -\pi\sum_{j=1}^N d_j \delta_{a_j}\, \, \textrm{ on } \partial \Om.$$ If $\phi_*$ is the harmonic extension to $\Omega$ of $\phi$, then $m_*=e^{\ci \phi_*}$ is a {\bf canonical harmonic map} associated to $\{(a_j, d_j)\}$.
\end{df}

\begin{rem}
\label{rem-uniq}
Note that the function $\phi$ in Definition \ref{defi:can_map} exists on $\dOm$ because of the contraint $\sum_{j=1}^N d_j=2$ that is equivalent via Gauss-Bonnet formula to the zero total mass of the (signed) measure $\partial_\tau \phi$, i.e., $\int_{\dOm} \partial_\tau \phi=0$. Moreover, the function $\phi$ is uniquely determined on $\dOm$ up to an additive constant in $\pi \ZZ$. Thus, for every prescribed boundary vortices $\{a_j\}_{1\leq j\leq N}$ with multiplicities 
$\{d_j\}_{1\leq j\leq N}$, there are {\bf exactly two canonical harmonic map} associated to $\{(a_j, d_j)\}$, i.e., $m_*=\pm e^{\ci \phi_*}$ for the harmonic extension $\phi_*$ to $\Omega$ of $\phi$.
\end{rem}

We prove the following characterization of canonical harmonic maps. Compared to the results in the seminal book of Bethuel-Brezis-H\'elein \cite{BethuelBrezisHelein:1994a}, the novelty here consists in dealing with the constraint that the canonical harmonic maps are tangent to the boundary away from the prescribed boundary vortices.

\begin{thm}\label{pro:canm}
Let $B_1$ be the unit disk, $\{a_j\}_{1\leq j\leq N}$ be $N\geq 1$ distinct points on  $\partial B_1$ with multiplicities 
$d_1, \dots, d_N\in\ZZ\setminus\{0\}$ such that $\sum_{j=1}^N d_j = 2$. Consider a point $b\in \partial B_1\setminus \{a_1, \dots, a_N\}$.
Then any canonical harmonic map with prescribed boundary vortices $\{(a_j, d_j)\}_{1\leq j\leq N}$ on $\partial B_1$ has the form
\be
\label{can_disk}
m_*(z)=\pm {\ci} b \prod_{j=1}^N \left(\frac{z-a_j}{|z-a_j|} \frac{|b-a_j|}{b-a_j} \right)^{d_j}, \quad \textrm{for all } \, \, z\in B_1.
\ee
If $\Omega\subset \R^2$ is a domain  such that\footnote{\label{foot1} By Kellogg's theorem, the existence of such $C^1$ conformal diffeomorphism holds for $C^{1, \alpha}$ simply connected bounded domains $\Omega\subset \R^2$, where $0<\alpha<1$.} there exists 
a $C^1$ conformal diffeomorphism $\Phi:\overline{B_1}\to \overline{\Omega}$ with inverse $\Psi$, then
\be
\label{can_gen}
M_*(w)=m_*(\Psi(w)) \frac{\Phi'(\Psi(w))}{|\Phi'(\Psi(w))|} \textrm{ for every } w\in \Omega,
\ee
is the canonical harmonic map with prescribed boundary vortices $\{(\Phi(a_j), d_j)\}$ on $\dOm$. 
\end{thm}

If in the formula \eqref{can_disk} in the unit disk $B_1$, we let $b\to a_1$, then $\frac{b-a_1}{|b-a_1|}\to \pm {\ci a_1}$ implying that any canonical harmonic map with prescribed boundary vortices $\{(a_j, d_j)\}_{1\leq j\leq N}$ on $\partial B_1$ has the form\footnote{For $N=1$ (so, $d_1=2$), we use the convention that $\prod_{j\neq 1} \left(\frac{a_1-a_j}{|a_1-a_j|}\right)^{-d_j}=1$  as an empty product, thus,  $m_*(z)=\pm \frac{\ci}{a_1} \left(\frac{z-a_1}{|z-a_1|}\right)^{2}$. }
$$m_*(z)=\pm ({\ci} a_1)^{1-d_1}\prod_{j=1}^N \left(\frac{z-a_j}{|z-a_j|}\right)^{d_j} \prod_{j\neq 1} \left(\frac{a_1-a_j}{|a_1-a_j|}\right)^{-d_j}\quad \textrm{for every} \quad z\in B_1$$ 
(similar formulas are obtained when $b$ tends to another boundary vortex $a_j$).
In particular, for $N=2$ boundary vortices $a\neq a'$ on $\partial B_1$ with multiplicities $d_1=d_2=1$, the canonical map is 
\[
m_*(z) = \pm \frac{(z-a)(z-a')|a-a'|}{|z-a||z-a'|(a-a')} \quad \textrm{for every} \quad z\in B_1.
\]

\bigskip

\nd {\bf Renormalized energy}. The interaction energy between boundary vortices is englobed in the so-called renormalized energy that we define in the following for multiplicities $d_j\in \{\pm 1\}$. This is a natural constraint on the multiplicities $d_j$ appearing in Ginzburg-Landau type functionals when computing the exact second order expansion (in the sense of $\Gamma$-convergence).

\begin{df}\label{defi:renen}
Let the transversal section $\Omega\subset \R^2$ be a bounded, simply connected, $C^{1,1}$ regular domain and $\ka$ be the curvature on $\dOm$.  Consider
$\phi:\partial\Omega\to \R$ to be a
 $BV$ function such that  $e^{\ci\phi}\cdot \nu=0$ on $\partial\Omega\setminus\{a_1,\dots,a_N\}$ for $N\geq 2$ distinct points $a_j\in \dOm$  and\footnote{Such a function $\phi$ exists  and is uniquely determined on $\dOm$ up to an additive constant in $\pi \ZZ$, see Remark \ref{rem-uniq}.}
$$\partial_\tau \phi =\ka  -\pi\sum_{j=1}^N d_j \delta_{a_j}\, \, \textrm{ on } \partial \Om \quad \textrm{with}\quad d_j\in \{\pm 1\} \textrm{ and }\sum_{j=1}^N d_j=2.$$ If $\phi_*$ is the harmonic extension to $\Omega$ of $\phi$, then the \emph{\bf renormalized energy} of $\{(a_j,d_j)\}$ is
\begin{equation}\label{eq:renW}
W_\Omega(\{(a_j,d_j)\}) =  \lim_{\rho\to 0} \left( 
\int_{\Omega \setminus \bigcup_{j=1}^N B_\rho(a_j)} |\nabla \phi_*|^2 \,dx - N\pi \log\frac1\rho
\right),
\end{equation}
where $B_\rho(a_j)$ is the disk of radius $\rho$ centered at $a_j$. 
\end{df}

We prove the following formula of the renormalized energy:

\begin{thm}\label{pro:ree}
Let $B_1$ be the unit disk, $\{a_j\}_{1\leq j\leq N}$ be $N\geq 2$ distinct points on  $\partial B_1$ with multiplicities 
$d_1, \dots, d_N\in\{\pm 1\}$ such that $\sum_{j=1}^N d_j = 2$. Then
the renormalized energy defined in \eqref{eq:renW} satisfies 
\[
W_{B_1}(\{(a_j, d_j)\}) = -{2\pi \sum_{1\leq j< k\leq N}} d_j d_k \log|a_j-a_k|.
\]
If $\Omega\subset \R^2$ is a {$C^{1,1}$ simply connected bounded }domain, let  
 $\Phi:\overline{B_1}\to \overline{\Omega}$ a $C^1$ conformal diffeomorphism with inverse $\Psi$. Then for any $N\geq 2$ distinct points $\{a_j\}_{1\leq j\leq N}$ on  $\dOm$ with multiplicities 
$d_1, \dots, d_N\in\{\pm 1\}$ such that $\sum_{j=1}^N d_j = 2$, the renormalized energy is given by
\begin{multline*}
W_{\Omega}(\{(a_j, d_j)\}) = -{2\pi  \sum_{1\leq j< k\leq N} }d_j d_k \log |\Psi(a_k) -\Psi(a_j)| +\sum_{k=1}^N \pi (d_k-1) \log |\Psi'(a_k)|
\\+
 \int_{\partial\Omega} \varkappa \bigg( \sum_{j=1}^N d_j \log|\Psi(w)-\Psi(a_j)|-\log|\Psi'(w)|\bigg) d\h^1,
 \end{multline*}
where $\varkappa$ is the curvature of $\partial \Omega$. 
\end{thm}

At the minimal level, we will prove that the energy functional $E_h$ concentrates asymptotically on two boundary vortices of multiplicities $1$. To locate these two boundary vortices, the following result is essential:

\begin{cor}
\label{cor-ren-min}
Let $\Omega\subset \R^2$ be a bounded, simply connected, $C^{1,1}$ regular domain. Then there exists a minimizer $(a_1^*, a_2^*)$ of the renormalized energy (for the multiplicities $d^*_1=d^*_2=1$): 
\be
\label{minim_ren}
 W_\Om(\{(a^*_1, 1), (a^*_2,1)\}) = \min \bigg\{W_\Omega(\{(a_1,1), (a_2,1)\})\, :\,  a_1,  a_2 \in \partial \Omega \textrm{ distinct points}\bigg\}.
 \ee
In particular, if $\Om=B_1$, then $a_1^*$ and $a_2^*$ are diametrically opposed and
 $W_\Om(\{(a_1^*,1),(a_2^*,1)\})=-2\pi \log 2$. 
\end{cor}

\begin{rem}
\label{rem_min_ellipse}
Let us formally analyse minimizing configurations in \eqref{minim_ren} for a general bounded $C^{1,1}$ simply connected domain $\Omega$ with curvature $\varkappa$ on the boundary $\partial \Omega$. If $\Phi:\overline{B_1}\to \overline{\Omega}$ is a $C^1$ conformal diffeomorphism with inverse $\Psi$, setting $b_1=\Psi(a_1)\in \partial B_1$, $b_2=\Psi(a_2)\in \partial B_1$ for two distinct points  $a_1, a_2\in \dOm$ with $d_1=d_2=1$, then Theorem \ref{pro:ree} implies
\begin{multline*}
W_\Omega(\{(a_1,1), (a_2,1)\})\\ = -2\pi \log|\Psi(a_1)-\Psi(a_2)| + \int_{\partial\Omega } \varkappa \bigg( \log |\Psi(w)-b_1| + \log|\Psi(w)-b_2| -\log |\Psi'(w)| \bigg) d\h^1.
\end{multline*}
After the change of variable $z=\Psi(w)$, we obtain 
\begin{multline*}
W_\Omega(\{(a_1,1), (a_2,1)\})\\=
-2\pi \log |b_1-b_2| + \int_{\partial B_1} \varkappa( \Phi(z)) |\Phi'(z)| \bigg( \log |z-b_1|+\log|z-b_2| + \log |\Phi'(z)|\bigg) d\h^1(z).
\end{multline*}
Then any minimal configuration $(a_1^*, a_2^*)$ in \eqref{minim_ren} corresponds
to points $b_1^*=\Psi(a_1^*)$ and $b_2^*=\Psi(a_2^*)$ that are expected to be the furthest apart and for which  the curvature $\varkappa$ at $a_1^*$ and $a_2^*$ is the highest (as $ \log |z-b_j^*|$ is negative for $z$ close to $b_j^*$, $j=1,2$), but there is a nontrivial 
competition between these effects through the influence of the conformal map. 

In particular, for an ellipse domain $\Omega$, $a_1^*$ and $a_2^*$ are expected to be placed at the two poles of major axis as this configuration maximes the diameter and the curvature of $\partial \Om$. {Also, if $\Om$ is a ``smoothed" rectangle (i.e., every corner is replaced by a convex $C^{1,1}$ curve), then the two boundary vortices $a_1^*$ and $a_2^*$ are expected to be diagonally opposed (so called $S$-state) as again this configuration maximes the distance $|a_1^*-a_2^*|$ and the curvature of $\partial \Om$.
We refer to \cite{Baffetti} for a more detailed discussion of the situation in rectangles and computations of the relevant energies.

This scenario is similar to the one analysed by Ignat-Jerrard \cite{IgJeP} in a Ginzburg-Landau model for tangent vector fields on a two-dimensional Riemannian manifold: on surfaces of genus $0$, two vortices of degree one nucleate and the optimal position of such a pair of vortices is expected  to maximes the distance between the vortices and the Gauss curvature of the surface.   }
\end{rem}

\bigskip

\nd {\bf $\Gamma$-expansion}. We can now state the main theorems of the paper, which contain several compactness results and 
two orders of energy expansion by $\Gamma$-convergence. We start with compactness and lower bound.
\begin{thm}\label{thm:t1}
Let $\Omega\subset \R^2$ be a bounded, simply connected, $C^{1,1}$ regular domain. 
If $h\to 0$, $\eta=\eta(h)\to 0$ and $\eps=\eps(h)\to 0$ satisfy the regime \eqref{regime}, then the following holds: 
Assume  $\bfm_h\in H^1(\Bfom_h; \Ss^2)$ is a sequence of magnetizations such that
$$\limsup_{h\to 0} E_h(\bfm_h)<\infty$$ with $E_h$ defined in \eqref{eee} and let $\bm=(\bar m_h, \bar m_{h,3})$ be the average defined in \eqref{eq:averageofm}. 

\begin{enumerate}[(i)]
\item \textbf{Compactness of the global Jacobian and of the traces $\bm\big|_{\partial \Om}$.} For a subsequence, the global Jacobians of the in-plane averages $\jaco(\bar m_h)$ converge  
 to a measure $J$ on the closure $\bar \Omega$,
in the sense that\footnote{This quantity is stronger than the usual $W^{-1,1}(\Omega)$-norm because our test functions in \eqref{conv_lip} are not necessarily zero on the boundary $\partial \Om$. } 
 \be
  \label{conv_lip}
  \sup_{|\nabla \zeta|\leq 1\textrm{ in } \Omega}\left|\left<\jaco(\bar m_h)-J,\zeta\right>\right|\to 0 \quad \textrm{as }\, h\to 0,
  \ee
 $J$ is supported on $\dOm$ and has the form 
\be
\label{newlab}
J=-\ka \, {\h^1\llcorner \dOm}+\pi \sum_{j=1}^N d_j \delta_{a_j} \quad \textrm{with} \quad  \sum_{j=1}^N d_j=2
\ee
for $N\geq 1$ distinct boundary vortices $a_j \in \partial\O$ carrying the non-zero multiplicities\footnote{These multiplicities correspond to twice the winding number around the boundary vortices.}
$d_j\in\ZZ\setminus \{0\}$. 
{Moreover, for a subsequence, the trace of the averages $\bm\big|_{\partial \Om}$ converges as $h\to 0$ in $L^p(\dOm)$ ({for every $p\in [1, \infty)$}) to $(e^{i\phi},0)\in  BV(\dOm; \Ss^1\times\{0\})$  
for a $BV$ lifting $\phi$ of the tangent field $\pm \tau$ on $\dOm$ determined (up to a constant in $\pi \ZZ$) by $\partial_\tau \phi =-J$ on $\dOm$.}

\item \textbf{First order lower bound.} The energy satisfies
\[
\liminf_{h\to 0} E_h(\bfm_h) \ge \pi\sum_{j=1}^N |d_j|.
\]
\item \textbf{Single multiplicity and second order lower bound.} If additionally {$\frac{\log|\log h|}{|\log h|}\ll \eps$} and
\[
\limsup_{h\to 0} |\log \eps| \bigg(E_h(\bfm_h) - \pi\sum_{j=1}^N |d_j|\bigg)<\infty, 
\]
then 
the multiplicities satisfy $d_j=\pm 1$ for $1\leq j\leq N$ and
 the finer
energy lower bound holds:
\[
\liminf_{h\to 0} |\log \eps| \bigl(E_h(\bfm_h) - \pi N \bigr) 
\ge W_\Omega(\{(a_j,d_j)\}) + {\gamma_0}N, 
\]
where $\gamma_0=\pi\log\frac{ e}{4\pi}$ is a  constant and the renormalized
energy $W_\Om$  is defined in \eqref{eq:renW}.

\item \textbf{Strong compactness of the {rescaled} magnetization.} Under the assumptions in (iii), we also have for every $q\in[1,2)$ the bound
\[
\limsup_{h\to 0}\frac1h\int_{\Bfom_h} |\bfa \bfm_h|^q \, dx <\infty.
\]
For a subsequence we have that  $\tilde \bfm_h(x,x_3):\Bfom_1\to \Ss^2$ defined by 
 $\tilde \bfm_h(x,x_3)=\bfm_h(x,hx_3)$ converges strongly in every $L^p(\Bfom_1)$, $1\le p<\infty$,  to a $W^{1,q}$-map  $\tilde\bfm=(\tilde m, 0)$ with
$|\tilde\bfm|=|\tilde m|=1$ and $\partial_{x_3}\tilde\bfm=0$, i.e., $\tilde \bfm=\tilde \bfm(x)\in W^{1,q}(\Omega, \Ss^1\times \{0\})$ for every $q\in[1,2)$. Moreover, the global Jacobian\footnote{The global Jacobian $\jaco(\tilde m)$ is 
well-defined as $\tilde m\in W^{1,1}(\Omega, \Ss^1)$, compare with footnote~\ref{fn3}.} $\jaco(\tilde m)$ coincides with the measure $J$ on $\bar \Omega$ given in \eqref{newlab}.
\end{enumerate}
\end{thm}

 We have a corresponding upper bound statement constructing a recovery sequence:
\begin{thm}\label{thm:t1UB}   
Let $\Omega\subset \R^2$ be a bounded, simply connected, $C^{1,1}$ regular domain, $h>0$ and $\eta=\eta(h)>0$ satisfying regime \eqref{regime}. Given any collection of $N\geq 1$ distinct points $\{a_j\in \partial \O\}_{1\leq j \leq N}$ and $\{d_j\in\ZZ\setminus\{0\}\}_{1\leq j \leq N}$
with $\sum_{j=1}^N d_j=2$,
 we can 
find $\bfM_h=(M_h,0)\in H^1(\Bfom_h;\Ss^1)$ with the following properties:
\begin{itemize}
\item[(i)] $\bfM_h$ is independent of $x_3$, i.e.
$\partial_{x_3} \bfM_h=0$ in $\Bfom_h$.
\item[(ii)] For any $0<x_3<h$, 
the global Jacobians $\jaco(M_h(\cdot, x_3))$ of the in-plane averages $M_h(\cdot, x_3)$ converge in the sense of \eqref{conv_lip} to
$J=-\ka{\h^1\llcorner\dOm}+\pi \sum_{j=1}^N d_j\delta_{a_j}$ as $h\to 0$, the average $\bar \bfM_h$ converges to $(m_*,0)$ in $L^p(\Omega)$ and $L^p(\partial \Omega)$ for every $p\in [1, \infty)$, where $m_*$ is a canonical harmonic map associated to $\{(a_j, d_j)\}$
and
\[
\lim_{h\to 0} E_h(\bfM_h) = \pi \sum_{j=1}^N |d_j|.
\]
\item[(iii)]
If furthermore $|d_j|=1$ for all $j=1,\dots,N$ and the narrower regime \eqref{regime2} holds,
 then $\bfM_h$ satisfies
\[
\lim_{h\to 0} |\log\eps| (E_h(\bfM_h)-\pi N) =  W_\Omega(\{(a_j,d_j)\})+ N\gamma_0.
\]
\end{itemize}
\end{thm}

\bigskip

The results of Theorems \ref{thm:t1} and \ref{thm:t1UB} generalize those we obtained \cite{IK_jac} for the reduced energy $\eeeh$ (restated as Theorem~\ref{thm:gammacforee} below), and 
in fact our proof uses Theorem~\ref{thm:lowerb} to connect these results.

\bigskip

By standard properties of $\Gamma$-convergence, we can immediately deduce 
\begin{cor}\label{cor:min} 
 If $\Omega\subset \R^2$ is a bounded, simply connected, $C^{1,1}$ regular domain, $h\to 0$, $\eta=\eta(h)\to 0$ and $\eps=\eps(h)\to 0$ satisfy the regime \eqref{regime}
and $\bfm_h$ are minimizers of $E_h$ as defined in \eqref{eee}, then the following holds: 
There exists a subsequence $h\to 0$ such that the global Jacobians $\jaco(\bar m_h)$ of the in-plane averages $\bar m_h$ satisfy
\[
\jaco(\bar m_h)\to 
J=-\ka \, \h^1\llcorner \dOm+\pi (\delta_{a^*_1}+\delta_{a^*_2})
\]
in the sense of \eqref{conv_lip}, for $a^*_1,a^*_2\in \dOm$ 
and the energy satisfies
\[
\lim_{h\to 0} E_h(\bfm_h) = 2\pi. 
\]
If additionally the assumption $\frac{\log|\log h|}{|\log h|}\ll \eps$ is satisfied, then $a_1^*\neq a_2^*$, the pair $(a_1^*,a_2^*)$ minimizes\footnote{For the existence of minimizers of $W_\Omega$, recall Corollary~\ref{cor-ren-min}.} $W_\Om(\{(a_1,1),(a_2,1)\})$ over the set $\{(a_1,a_2)\in \partial\Omega\times\partial\Omega\, :\, a_1\neq a_2\}$ and
\[
\lim_{h\to 0} |\log\eps| (E_h(\bfm_h) - 2\pi) = W_\Om(\{(a_1^*,1),(a_2^*,1)\})+2\pi \log\frac{e}{4\pi};
\]
furthermore, the sequence $\tilde \bfm_h(x,x_3):\Bfom_1\to \Ss^2$ defined by 
 $\tilde \bfm_h(x,x_3)=\bfm_h(x,hx_3)$ converges strongly in every $L^p(\Bfom_1)$, $1\le p<\infty$,  to an $x_3$-invariant $W^{1,q}$-map  $\tilde\bfm=(m_*, 0)$ with $m_*$ is a canonical harmonic map associated to $\{(a_1^*,1), (a_2^*,1)\}$ in $\Omega$. 
\end{cor}

\begin{rem} 
In the unit disk $\Om=B_1$, by Theorem~\ref{pro:ree} and Corollary \ref{cor-ren-min}, the renormalized energy for two boundary vortices of multiplicities $1$ has the form  
 $W_{B_1}(\{(a_1,1),(a_2,1)\})=-2\pi \log|a_1-a_2|$ and any minimal configuration is given by two diametrically opposed points $a_1^*$ and $a_2^*$ on $\partial B_1$ yielding
 $W_{B_1}(\{(a_1^*,1),(a_2^*,1)\})=-2\pi \log 2$. Together with  Corollary~\ref{cor:min}, we obtain 
 \[
 \lim_{h\to 0} |\log\eps| (E_h(\bfm_h) - 2\pi) = -2\pi \log 2+2\pi \log\frac{e}{4\pi}=2\pi \log\frac{e}{8\pi}.
 \]
 \end{rem}

\begin{rem}
Theorem~\ref{thm:t1} (iii) and Corollary \ref{cor:min} suggest that for minimizers $\bfm_h$ of $E_h$, no higher degree transitions can occur,
as the limit only shows singularities of multiplicity $1$.
This is similar to results for interior Ginzburg-Landau vortices \cite{BethuelBrezisHelein:1994a}. Generalizing from minimizers to
 critical points, the situation appears fundamentally different
between boundary and interior vortices, as can be seen from the (blow up) results in the whole plane or the half plane: 
For Ginzburg-Landau vortices, (unstable) solutions of higher degrees were shown to exist by Herv\'e--Herv\'e~\cite{HerveHerve:1994a} and Chen--Elliott--Qi \cite{ChenElliottQi:1994a}. In the 
{boundary vortex} case, solutions on a half plane can only have a single transition by results of Toland~\cite{Toland:1997a} and Cabr\'e--Sol\`a-Morales~\cite{CabreSola-Mor:2005a}. 
By recent results of Baffetti at al. \cite{BEK19} for critical points of an $\Ss^1$-valued model of boundary vortices on a bounded domain, 
it is impossible for these transitions to cluster at distances that are $\gg \eps$, but $\ll 1$, so the limit can only have singularities of multiplicity $1$.
\end{rem}

\begin{rem}\label{rem:multiconn}
We have required that the domain is simply connected. The results are \textbf{false} for doubly connected domains, but 
analogous results to ours  can be expected to hold for domains of higher connectivity. Doubly connected domains (like an annulus) are special because they support continuous vector fields that are tangential to the boundary \textbf{everywhere}, and there are even examples of smooth magnetizations on such domains for which the stray field
vanishes,
as noted in the physics literature by Arrott et al. \cite{Arrott}.
\end{rem}

\subsection*{Outline of the paper} The rest of the paper is organized as follows: In the next section, we
describe a panorama of thin film regimes and the main features of interior and boundary vortices in micromagnetics. 
In Section~\ref{sec:reduc}, we reduce
 the nonlocal $3D$ energy $E_h$ energy to
the simplified local functional $\eeeh$ in \eqref{eq:defofebar}
by showing that these energies are close to each other 
up to $o(\frac1{|\log\eps|})$ (or $o(1)$ for the
highest order of $\Gamma$-development), see Theorem~\ref{thm:lowerb}.
This is done by a careful series of estimates that refine
results of Gioia-James
\cite{GioiaJames:1997a}, Carbou \cite{Carbou:2001a} and Kohn-Slastikov \cite{KohnSlastiko:2005a} in a more quantitative way. In Section \ref{sec:sec3}, we prove Theorem~\ref{thm:t1} using 
the analysis of the simplified local functional $\eeeh$ in \eqref{eq:defofebar} from our companion article \cite{IK_jac}. In Section \ref{sec:ren}, we prove the properties of the canonical harmonic maps and the renormalized energy stated in Theorems \ref{pro:canm} and \ref{pro:ree} as well as Corollary \ref{cor-ren-min}. We end with an appendix proving the characterization of the stray field in \eqref{Helmholtz}.

\subsection*{Acknowledgments} The research presented in this article has been supported by ANR project MAToS, no. ANR-14-CE25-0009-01 
and DFG SFB 611.

\section{Related models. Micromagnetic vortices} 
\label{physi}

Our choice of thin film regime \eqref{regime} is not the only one that leads to a thin film $\Gamma$-limit. In fact,
there is a whole range of possible limits, and we give a short panoramic overview here, see Figure~\ref{fig:regimes}. 
We always assume $h\ll 1$ which corresponds to thin film regimes. 
The most obvious thin film limit (of letting  $h\to 0$) corresponds to small magnetic samples where $\eta>0$ is fixed (i.e., $A\sim \ell$); this regime  has been studied
by Gioia-James~\cite{GioiaJames:1997a} (see also Kreisbeck \cite{Kreisbeck:2013aa} for an alternate approach). 
The resulting $\Gamma$-limit is somewhat degenerate in 
the sense that it is minimized by all constant in-plane magnetizations and does not depend on  the shape of the film.
Recently, Morini-Slastikov~\cite{Morini:2018aa} 
also studied small films with additional surface roughness and were able to derive a homogenized 
thin film limit, with constant minimizers determined by the geometry of the roughness.

The case of larger magnetic samples $\eta\ll 1$, i.e., the exchange length $A$ is small compared with the diameter $\ell$, is more important as it is physically achievable. There are different convergence rates of $\eta(h)\to 0$ as $h\to 0$ 
corresponding to samples of various size and leading to various regimes where different effects come into play. Three types of singular pattern of
the magnetization occur: N\'eel walls, interior and boundary micromagnetic vortices. The choice of the asymptotic regimes will correspond to the energy ordering of these three patterns (for more details, see \cite{DeSimoneKohnMuller:2005a}).
All of the regimes we study are separated only by logarithmic (or even doubly logarithmic) terms. For this reason, the sharp separation of regimes
found by $\Gamma$-convergence is more  prominent in the analysis than in physical or numerical experiments at finite sample sizes.
We  list some regimes and the related results by increasing sample size, corresponding to decreasing $\eta$. 

\begin{figure}[htbp]
  \begin{tikzpicture}[x=2mm,y=5mm]
\draw[thick, -Triangle] (5,0) -- (60,0) node[font=\small,below left=3pt and -8pt]{$\eta^2$};

\foreach \x in {1,...,5}
\draw (10*\x,-0.2) -- (10*\x,0.2);

\node[font=\small, text height=1.75ex,
text depth=.5ex] at (0,0.3) {$A$ small or};
\node[font=\small, text height=1.75ex,
text depth=.5ex] at (0,-0.3){$\ell$ large};
\node[font=\small, text height=1.75ex,
text depth=.5ex] at (65,0.3) {$A$ large or};
\node[font=\small,text height=1.75ex,
text depth=.5ex] at (65,-0.3){$\ell$ small};

\node[font=\small, text height=1.75ex,
text depth=.5ex] at (10,-0.7) {$\frac{h}{|\log h|}$};
\node[font=\small, text height=1.75ex,
text depth=.5ex] at (20,-0.7) {$\frac{h}{\log|\log h|}$};
\node[font=\small, text height=1.75ex,
text depth=.5ex] at (30,-0.7) {$h$};
\node[font=\small, text height=1.75ex,
text depth=.5ex] at (40,-0.7) {$h\log|\log h|$};
\node[font=\small, text height=1.75ex,
text depth=.5ex] at (50,-0.7) {$h|\log h|$};
     \mmregime{DKMO}{0}{1}{0}{10}{2}{lightgray}
    \mmregime{IKn}{0}{8}{1}{7}{3}{lightgray}
    \mmregime{IO}{1}{9}{2}{9}{3}{lightgray}
    \mmregime{Moser}{2}{6}{3}{4}{2}{lightgray}
        \mmregime{This paper: Th.\ref{thm:t1}(ii)}{3}{4}{4}{8}{2}{gray} 
        \mmregime{Th.\ref{thm:t1}(iii)}{4}{2}{4}{8}{3}{gray}
        \mmregime{KS}{4}{8}{5}{4}{2}{lightgray}
        \mmregime{KS/Carbou}{5}{4}{5}{12}{2}{lightgray}
        \mmregime{GJ}{5}{9}{5}{13}{3}{lightgray}
        \mmregime{K}{4}{8}{4}{11}{3}{lightgray}
 \end{tikzpicture}
\caption{An approximate panoramic view (not to scale) of thin film limits and their range of validity  (results of this article in dark grey, others in light grey: 
DKMO=DeSimone et al.~\cite{DesimoneKohnMuller:2002a}, IKn=Ignat-Kn\"upfer~\cite{Ignat:2010aa},
IO=Ignat-Otto~\cite{IgnatOtto:2011a}, Moser~\cite{Moser:2004a}, KS=Kohn-Slastikov~\cite{KohnSlastiko:2005a}, Carbou~\cite{Carbou:2001a},
GJ=Gioia-James~\cite{GioiaJames:1997a}).
The results of K=Kurzke \cite{Kurzke:2006a,Kurzke:2006b,Kurzke:2007a} can be interpreted as a limit at the left ``larger films'' end of the KS regime. }
\label{fig:regimes}
\end{figure}
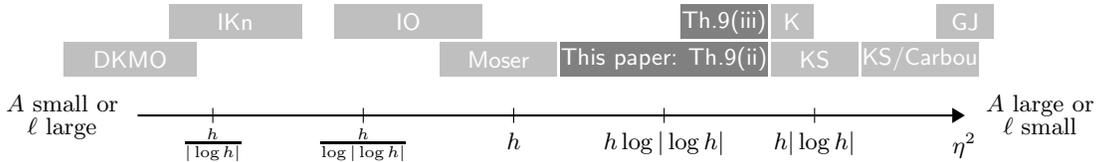

In the regime of relatively small films, characterized by 
$\eta^2\gg h{|\log h|}$, 
the exchange term in the energy
dominates completely and the magnetization becomes a constant unit-length
vector in the film plane.
A corresponding $\Gamma$-limit was derived by Kohn-Slastikov 
\cite{KohnSlastiko:2005a}, related to earlier work by Carbou
\cite{Carbou:2001a}. Their result is that
the nonlocal magnetostatic energy reduces to a local contribution
of the boundary $\frac1{2\pi}\int_{\partial\O} (m\cdot \nu)^2 d\h^1$, 
 which selects the preferred directions of the constant magnetization vector.

Slightly larger films, where 
$\eta^2=\alpha {h}{|\log h|}$
with $0<\alpha<\infty$, 
were
also studied by Kohn-Slastikov. Here, the limiting magnetizations
are still required to lie in the film plane, but no longer need to be
constant. Instead, the exchange energy and the boundary contribution
compete, and the rescaled energy $\Gamma$-converges to
\[
E^\alpha_{KS}(m)=  \alpha\int_\O |\nabla m|^2dx + \frac1{2\pi }\int_{\partial\O}(m\cdot \nu)^2 d\h^1, \quad m\in H^1(\O;\Ss^1).
\]
A second limit, describing the behavior of 
$\frac1\alpha E^\alpha_{KS}$ when $\alpha\to 0$, 
was examined by Kurzke \cite{Kurzke:2006a,Kurzke:2006b,Kurzke:2007a}. 
There is no 
$m\in H^1(\O;\Ss^1)$ that satisfies $m\cdot \nu=0$ on $\partial\O$ 
if $\O$ is a simply connected domain. For this reason, the boundary term cannot be made zero, and 
 for small $\alpha$ we obtain the emergence of boundary vortices, 
where the magnetization quickly rotates from $m\approx \tau$ to $m\approx -\tau$ over a boundary segment of length $O(\alpha)$
(see Section \ref{sec:topo} for further details).
The highest order term in the energy expansion relates to the number of boundary vortices, 
while their
 interaction is
  governed by a
renormalized energy appearing as the next order term in the energy expansion as $\alpha\to 0$. However,
the significance of these results in the context of the full micromagnetic energy remained unclear for a long time; 
 the main purpose of the present paper is to clarify this.

Our Theorem \ref{thm:t1} in the present article 
directly relates to the micromagnetic energy and shows that the double limit procedure of Kohn-Slastikov and Kurzke
yields  the correct result for the highest order of the energy and its concentration at boundary vortices  if we are in the regime $h\ll \eta^2 \ll h|\log h|$. 
In the narrower regime $h\log|\log h|\ll \eta^2 \ll h |\log h|$, we obtain the same renormalized energy as Kurzke.

The next regime,  $\eta^2=O(h)$, was studied by 
Moser \cite{Moser:2003a,Moser:2004a,Moser:2005a}. Here, both the magnetostatic and 
exchange terms survive in the limit, and again, an incompatibility
produces boundary vortices. To highest order, the energy is proportional to the number of vortices.
 The corresponding boundary vortex interaction is 
nonlocal here, in contrast to the local renormalized energy of the present article. 
For a review of these models, we refer to Kurzke-Melcher-Moser \cite{KurzkeMelcherMoser:2006a}.

In large thin films corresponding to the regime $\eta^2\ll h$, N\'eel walls nucleate in the magnetic sample. The (symmetric) N\'eel wall is a transition layer describing a one-dimensional in-plane rotation connecting two (opposite) directions of the magnetization. It is a two-length scale object with a core of size of order $\frac{\eta^2}{h}$ and two logarithmically decaying tails (see \cite{Mel1, Mel2, DKMO-Neel, IM2016}). Various aspects of N\'eel walls (existence of topological N\'eel walls with prescribed winding number, interaction between N\'eel walls, cross-over from symmetric to asymmetric N\'eel walls etc.)  have been analyzed recently
(see e.g. \cite{DIO, DI-asym, IgGamma, IM2016, IM2017, IM2019}). 

We now describe briefly three sub-regimes  for the limit $\eta^2\ll h$ where the nonlocality of the reduced energy is carried by the $H^{-1/2}$ norm of the volume charges $\nabla \cdot m$ inside $\Om$ yielding the highest order energy of N\'eel walls (for more details, see section 7.2. in \cite{Ignat_HDR}).
The sub-regime $\eta^2 |\log \frac{\eta^2}{h}| \gg \frac{h}{\log|\log h|}$ yielding
$$\frac{h}{\log|\log h| \cdot \log \log |\log h|}\ll \eta^2 \ll h$$
was studied by Ignat-Otto \cite{IgnatOtto:2011a}: next to the nonlocal term, the reduced energy penalizes the out-of-plane component $m_3$. The constraint $m\cdot \nu=0$
is imposed (so no $\Ss^1$-valued boundary vortices nucleate in that model); thus, the Landau state is composed by N\'eel walls and topological point singularities where $m_3=\pm 1$ nucleating either in the interior or at boundary of $\Om$. 

Ignat-Kn\"upfer \cite{Ignat:2010aa} studied a further regime of slightly larger  films, characterized by 
$
\frac{h}{|\log h|} \ll \eta^2 |\log \frac{\eta^2}{h}| \ll \frac{h}{\log|\log h|},
$ yielding $$\frac{h}{|\log h| \cdot \log |\log h|}\ll \eta^2 \ll \frac{h}{\log|\log h| \cdot \log \log |\log h|}.$$
The model is described by $\Ss^1$-valued magnetizations, so the system nucleates N\'eel walls and boundary vortices. It is proved in  \cite{Ignat:2010aa} that the Landau state in circular thin film is given by a peculiar vortex structure, driven by a topological N\'eel wall of winding number $1$ accompanied by a pair of micromagnetic boundary vortices (so the S-state is less favorable in that model).

In very large films, characterized by  $\eta^2\ll \frac{h}{|\log h|}$,
the contribution of the exchange energy disappears completely, and one 
obtains a purely magnetostatic model driven by $\| \nabla \cdot m\|^2_{H^{-1/2}}$ where the constraint $|m|=1$ 
relaxes to $|m|\le 1$, see DeSimone et al. \cite{DesimoneKohnMuller:2002a,DeSimoneKohnMuller:2005a}.

\medskip

\subsection{Topological point defects}
\label{sec:topo}
We further present some heuristics on interior and boundary vortices that shows the importance of the global Jacobian. More details
can be found in \cite{IK_jac}.

 \medskip

 {\nd \bf Interior vortex}. The prototype of an interior vortex is given by a map
$$\bfm=(m, m_3):B_1\to \Ss^2$$ defined in a circular cross-section $\Omega=B_1$ of
 a thin film that minimizes the reduced energy $\eeeh$ defined in \eqref{eq:defofebar} under the boundary condition 
 \be
 \label{BC_tang}
 m=\tau\quad \textrm{ on }\quad \partial B_1,\ee where $\tau(x)=(-x_2, x_1)$ is a tangent vector at $x\in \partial B_1$. (Recall that $B_1$ is the unit disk in $\R^2$.) Thus, the magnetization turns in-plane at the boundary, in particular, $m$ carries a topological degree $\degr(m, \partial \Omega)=1$. As for the Ginzburg-Landau energy, a localized region is created at the center that is the core of the vortex of size $\eta$. The difference consists in the polarity carried by micromagnetic interior vortex according to the value $m_3=\pm 1$ at the vortex point given by the topologic zero of $m$ where the magnetization $\bfm$ becomes perpendicular to the horizontal plane (see Figure
 \ref{Blochli}). 
 \begin{figure}[htbp]
  \center  
  \includegraphics[scale=0.5,width=0.5\textwidth]{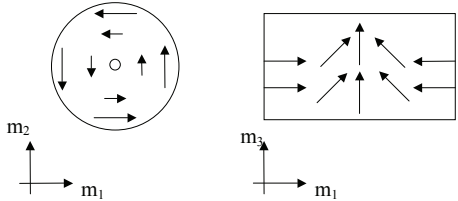}\\
  \caption{Micromagnetic interior vortex of winding number one and polarity one.} \label{Blochli}
\end{figure}
Note that the reduced energy $\eeeh$  (renormalized by $|\log \eps|$) controls the Ginzburg-Landau energy defined 
in \eqref{ginlan}, i.e.,
$$|\log \eps| \eeeh(\bfm)= \ \int_{\Omega} |\nabla \bfm|^2 \ dx + \frac{1}{\eta^2}\int_{\Om} (1-|m|^2) \ dx \geq \int_{\Omega} \den(m)\, dx$$
since
$|\nabla m|^2\leq |\nabla (m, m_3)|^2$ and $(1-|m|^2)^2=m_3^4\leq m_3^2=1-|m|^2$. 
Thus, the analogy with the theory of Ginzburg-Landau vortices (see \cite{BethuelBrezisHelein:1994a} and the review paper \cite[Section 7]{IgX}) yields:
$$
 \min_{\eqref{BC_tang}} \eeeh(\bfm)=\frac{2\pi|\log \eta|}{|\log \eps|}+O(\frac1{|\log \eps|})\quad \textrm{as}\quad \eps, \eta\to 0.
$$
As $m=\tau$ at the boundary, the boundary Jacobian of $m$ (defined in \eqref{def_jac_bd}) is carried by the curvature $\ka=1$ on  $\partial \Omega$ {(without any singular part)}, while the interior Jacobian of $m$ asymptotically concentrates on a Dirac measure ${\bf \delta}_0$ at the origin (up to a multiplicative constant); summing up, the global Jacobian of $m$ is given by
$$\jaco(m)=2\pi {\bf \delta}_0-\ka \, \h^1 \llcorner \partial \Om+o(1) \quad \textrm{as}\quad \eps, \eta\to 0.$$

\medskip

 {\nd \bf Boundary vortex.} The typical situation is given by an $\Ss^1$-valued map $m$ that minimizes the reduced energy in the unit disk $\Omega=B_1$ (i.e., $m_3=0$ in $B_1$):
 $$\eeeh(m)=\frac{1}{|\log \eps|} \bigg( \ \int_{\Omega} |\nabla m|^2 \ dx +\frac{1}{2\pi \eps} \int_{\partial \Om} (m\cdot \nu)^2\, d\h^1\bigg)  \quad \textrm{with } \, m:\Om\to \Ss^1,$$
 where $\nu=x$ on $\partial B^2$.  This problem has been analyzed by Kurzke \cite{Kurzke:2006a, Kurzke:2006b} and Moser \cite{Moser:2003a}: any minimizer $m$ is an harmonic map of unit length
driven by a pair of boundary vortices $P_1$ and $P_2$ that {are expected to be}  diametrically opposed of degree $1/2$ (see Figure \ref{fig:two_vort}). 
A boundary vortex of degree $1/2$ corresponds to an in-plane
transition of the magnetization $m$ along the boundary from $-\tau$ to $\tau=\nu^\perp$, i.e., the lifting of $m$ has an asymptotically jump of $-\pi$ (see Figure \ref{fig-bdrvortex}). 
The transition is regularized due to the exchange energy and the core of the boundary vortex has size $\eps$. The cost of such a transition is given by 
$$\frac 1 2 \eeeh(m)=\pi+O(\frac{1}{|\log \eps|}).$$ As $m$ is smooth of unit-length in $\Omega$, the interior Jacobian of $m$ vanishes so that the global Jacobian of $m$ is concentrated at the boundary: it is asymptotically given by a measure of regular part carried by the negative of the curvature, $-\ka=-1$ on $\partial \Omega$ and of singular part carried by two Dirac measures at $P_1$ and $P_2$:
$$\jaco(m)=\jacbd(m)=2\pi \bigg(\frac12\mathbb{\delta}_{P_1}+ \frac12\mathbb{\delta}_{P_2}\bigg)-\ka\, \h^1 \llcorner \partial \Om+o(1) \quad \textrm{as}\quad \eps, \eta\to 0.$$
Comparing with the interior vortex case, we see a justification for calling these ``half-degree'' vortices. However, for notational convenience we 
have written 
$\pi d_j \delta_{a_j}$ with $d_j\in\ZZ$ instead of $2\pi d_j \delta_{a_j}$ with $d_j \in \frac12\ZZ$ in the remainder of this paper.

  \begin{figure}
\center
  \includegraphics[scale=0.2,width=0.2\textwidth]{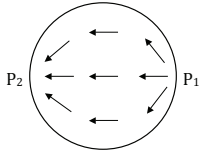}
 \caption{Two boundary vortices $P_1$ and $P_2$ of degree $1/2$. }
 \label{fig:two_vort}
      \end{figure}

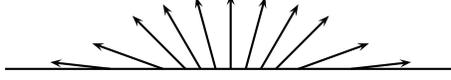
\begin{figure}[ht] %
 \center
  \begin{pspicture}(-3,0)(3,1) %
    \psset{unit=1cm}
    \psline(-3,0)(3,0)
    \psline{->}(0,0)(0, 1)
    \psline{->}(0.2,0)(0.46, 0.96) 
    \psline{->}(0.4,0)(0.9, 0.86)     
    \psline{->}(0.6,0)(1.307, 0.707) 
    \psline{->}(0.9,0)(1.84, 0.34) 
    \psline{->}(1.6,0)(2.4, 0.09) 
    \psline{->}(-0.2,0)(-0.46, 0.96) 
    \psline{->}(-0.4,0)(-0.9, 0.86)     
    \psline{->}(-0.6,0)(-1.307, 0.707) 
    \psline{->}(-0.9,0)(-1.84, 0.34) 
    \psline{->}(-1.6,0)(-2.4, 0.09) 
    \end{pspicture}
  \caption{Boundary vortex of degree $+1/2$
   in the upper half-plane $\R\times (0, \infty)$ within the frame $(\nu=-e_2, \tau=e_1)$. }
  \label{fig-bdrvortex}
\end{figure}

\medskip 

In our regime \eqref{regime} and with the energy scaling \eqref{eee},
 an ``essential" interior vortex will asymptotically induce infinite energy (because $|\log \eps|\ll |\log \eta|$ as $\eps, \eta\to 0$) while boundary vortices carry only finite energy. Therefore, we expect that no essential interior vortex nucleates for magnetizations of uniformly bounded energy so that our $\Ss^2$-valued magnetizations can be approximated by $\Ss^1$-valued maps (this expectation can be made rigorous using the estimates of \cite{IKL20},  see \cite[Theorem 3.1]{IK_jac})
 and the global Jacobian asymptotically concentrates at the boundary according to \eqref{newlab}.

\section{Reduction to a local model. Proof of Theorem \ref{thm:lowerb}}
\label{sec:reduc}
In this section we relate the nonlocal $3D$ micromagnetic energy $E_h$ for the magnetization $\bfm_h$ in \eqref{eee} to the simpler local $2D$ energy functional $\eeeh$ for the average $\bm$ in \eqref{eq:defofebar}.  The aim is to prove Theorem \ref{thm:lowerb}; for that, the key estimate is given in 
the following result. It is a more quantitative version of  estimates of Carbou \cite{Carbou:2001a} 
 and 
Kohn-Slastikov \cite{KohnSlastiko:2005a}.

\begin{lem}\label{lem:l1}  Let $\Bfom_h=\Omega\times (0,h)$ with $\Om\subset \R^2$ be a
simply connected
\footnote{This assumption is crucial: in fact, we use e.g. in the proof of 
Lemma \ref{lem:claim} that $\eeeh(\bm)\gtrsim 1$ as $h\to 0$ which may fail if the domain 
is topologically an annulus. Lower energy configurations are possible there, compare Remark~\ref{rem:multiconn}.}
 $C^{1,1}$ domain. In the regime \eqref{regime} of parameters $h, \eta(h), \eps(h)>0$, we consider a family of magnetizations 
 $\{\bfm_h:\Bfom_h\to \Ss^2\}_{h\to 0}$ and stray field potentials $\{U_h\}_{h\to 0}$ given by \eqref{strayfield} of 
uniformly bounded energy $\limsup_{h\to 0} E_h(\bfm_h)<\infty$. Then
we have the estimates as $h\to 0$:\footnote{The implicit constants in the big $O$ notation in \eqref{eq:bmhbd} and \eqref{eq:eandbare} depend only on $\Omega$.}
\begin{equation}\label{eq:bmhbd}
0\leq \frac{1}{\eta^2 |\log \eps|}\bigg[\int_{\Om} (1-|\bmp|^2)\, dx -\int_{\Om} \bmt^2\, dx \bigg]= \sqrt{\frac{E_h(\bfm_h)}{ |\log \eps|}} O(\frac{h}{\eta^2})
\end{equation}
and
\begin{equation}\label{eq:eandbare}
E_h(\bfm_h)\ge \eeeh(\bm) -  \bigg(\eeeh(\bm)+
   \sqrt{\frac{E_h(\bfm_h)}{|\log \eps|}}\bigg) O(\lambda(h)),
\end{equation}
where $\bm=(\bmp, \bmt)$ is the $x_3$-average of $\bfm_h$, $\eeeh$ is the reduced energy defined in \eqref{eq:defofebar},
and 
\be
\label{ah}
\lambda(h)=\frac{h}{\eta^2}\big(\frac{\log \frac{\eta^2}{h}}{|\log \eps|} 
+ 1\big) = \frac{1}{\eps|\log\eps|}\frac{\log|\log h|}{|\log h|}\ll 1.
\ee
Note that $\lambda(h)\ll 1$ is a direct consequence of \eqref{regimelogs}.
\end{lem}

\proof{} 
First, we prove the estimate \eqref{eq:bmhbd} for the distance of $\bm$ to the unit sphere $\Ss^2$. For that, denoting $\bfm_h=(m_{h,1}, m_{h,2}, m_{h,3})$ and recalling that $|\bm|\leq 1=|\bfm_h|$ in $\Bfom_h$, the Cauchy-Schwarz and Poincar\'e inequalities imply for $1\leq k\leq 3$:
\begin{align*}
\int_{\Bfom_h} |m_{h,k}^2(x,x_3)-\mk^2(x)|\, d\bfx
&\leq 2\int_\Om dx \int_0^h |m_{h,k}(x,x_3)-\mk(x)|\, dx_3\\
&\leq 2 \sqrt{h} \int_{\Om} \, dx \bigg( \int_0^h |m_{h,k}(x,x_3)-\mk(x)|^2\, dx_3\bigg)^{1/2}\\
&\leq C h^{3/2} \int_{\Om} \, dx \bigg( \int_0^h |\partial_{x_3} m_{h,k}(x,x_3)|^2\, dx_3\bigg)^{1/2}\\
&\leq  C h^{3/2} \bigg( \int_{\Bfom_h} |\partial_{x_3} m_{h,k}(\bfx)|^2\, d\bfx\bigg)^{1/2},
\end{align*}
for $C>0$ is a constant depending only on $\Omega$ that changes from line to line.
Since $|\bfm_h|=1$, summing for $k=1, 2, 3$, we deduce
\be
\label{estin12}
\int_{\Om} (1-|\bm|^2)\, dx \leq C h (|\log \eps| E_h(\bfm_h))^{1/2},
\ee
which leads to \eqref{eq:bmhbd} because $0\leq 1-|\bm|^2=(1-|\bmp|^2)-\bmt^2$.

For the second inequality, we start by noting that Jensen's inequality implies
$$\int_{\Omega} |\nabla \bm|^2 \ dx\leq \frac{1}{h} \int_{\Bfom_h} |\bfa \bfm_h|^2 \ d\bfx. $$
To estimate the stray field energy, we use the strategy of Kohn-Slastikov (see \cite{KohnSlastiko:2005a}, Lemma 3). We denote by $\bu\in H^1(\R^3)$ the stray field potential
associated to $\bm { \mathds 1}_{\Bfom_h}$ in \eqref{stray_mean}.
The definition of Helmholtz projection \eqref{Helmholtz} together with the Poincar\'e inequality lead to
\begin{align*}
& \int_{\RR^3} | \bfa U_h|^2 \ d\bfx\leq \int_{\Bfom_h}|\bfm_h|^2\, d\bfx\leq C h, \quad \int_{\RR^3} | \bfa \bu|^2 \ d\bfx
\leq \int_{\Bfom_h}|\bm|^2\, d\bfx\leq C h, \\
&\int_{\RR^3} | \bfa U_h-\bfa \bu|^2 \ d\bfx\leq \int_{\Bfom_h}|\bfm_h-\bm|^2\, d\bfx\leq C h^2 \int_{\Bfom_h} |\partial_{x_3}\bfm_h |^2\, d\bfx.
\end{align*}
Since $\big|\|a\|^2-\|b\|^2\big|\leq \big(2\|a-b\|^2(\|a\|^2+\|b\|^2)\big)^{1/2}$, we finally obtain:   
\begin{align*}\frac{1}{\eta^2 h |\log \eps|}\bigg|\int_{\RR^3} | \bfa
   U_h|^2 \ d\bfx-\int_{\RR^3} | \bfa
   \bu|^2 \ d\bfx \bigg|&\leq \frac{Ch}{\eta^2 |\log \eps|} \bigg(\frac{1}{h}\int_{\Bfom_h} |\partial_{x_3} \bfm_h|^2\, d\bfx\bigg)^{1/2}\\
   & \leq C\frac{h}{\eta^2} \sqrt{\frac{E_h(\bfm_h)}{|\log \eps|}} .\end{align*}
Next, we use Lemma~\ref{lem:claim} below to estimate the stray field energy generated by $\bu$. 
Then \eqref{eq:eandbare}
follows by using \eqref{eq:bmhbd}.
\qed

\medskip

The following result improves computations of  \cite{Carbou:2001a} and \cite{KohnSlastiko:2005a} 
 in our regime \eqref{regime}. 

\begin{lem}\label{lem:claim}
Under the assumptions in Lemma \ref{lem:l1}, if $\bu$ is the stray field potential associated to the $x_3$-average $\bm:\Omega\to \bar B^3$ in \eqref{stray_mean}, then we have for $h\to 0$:
\begin{equation}
\label{claim}
\frac{1}{|\log \eps|}\bigg|\frac{1}{\eta^2 h}\int_{\RR^3} | \bfa
   \bu|^2 \ d\bfx-\frac{1}{\eta^2}\int_{\Om} \bmt^2 \ dx-\frac{1}{2\pi \eps} \int_{\partial \Om} \bmn^2\, d\h^1\bigg| 
\leq C  \lambda(h) 
\eeeh(\bm)
   \end{equation}
where $ \lambda(h)>0$ is given in \eqref{ah} and $C>0$ is a constant depending on $\Omega$. 
\end{lem}

\proof{} 
   From \eqref{Helmholtz}, integration by parts yields
$$\int_{\RR^3} | \bfa
   \bu|^2 \ d\bfx=\int_{\Bfom_h} \bfa 
   \bu\cdot \bm \ d\bfx=-\int_{\Bfom_h} \bu \nabla \cdot  \bmp\, d\bfx +\int_{\partial \Bfom_h} \bu (\bm\cdot {\bfnu}) \ d\h^2(\bfx),$$
   where $\bfnu=(\nu, \nu_3)$ is the unit outer normal vector to $\partial \Bfom_h$. By Proposition \ref{pro:append}  in the Appendix, we have for every $\bfx\in \R^3$: 
$$4\pi \bu(\bfx)=-\int_{\Bfom_h} \frac{1}{|\bfx-\bfy|} \nabla\cdot  \bmp(y) \, d\bfy+ \int_{\partial \Bfom_h}\frac{1}{|\bfx-\bfy|} (\bm\cdot {\bfnu})(\bfy)\, d\h^2(\bfy).$$ 
Combining the above equalities, we obtain:
$$4\pi\int_{\RR^3} | \bfa
   \bu|^2 \ d\bfx={\cal A}+2{\cal B}+{\cal C}$$
where the terms ${\cal A}$ and ${\cal B}$ are estimated as in \cite[Lemma 1 and 2]{KohnSlastiko:2005a} using the generalized Young's inequality:\footnote{In particular,  if $f\in L^2(\Om)$ and $g\in L^2(\dOm)$, then $\int_\Om \int_{\dOm} \frac{f(x) g(y)}{|x-y|}\, dx d\h^1(y)\leq C\|f\|_{L^2(\Om)} \|g\|_{L^2(\dOm)}$. Indeed, denoting for $y\in \dOm$, $F(y)=\int_\Om \frac{f(x)}{|x-y|}\, dx$, H\"older's inequality implies $F^2(y)\leq \int_\Om \frac1{|x-y|^{3/2}}\, dx \int_\Om \frac{f^2(x)}{|x-y|^{1/2}}\, dx$ and thus, $\int_{\dOm} F^2(y)\, d\h^1(y)\leq c(\Om)  \|f\|^2_{L^2(\Om)} \sup_{x\in \Om} \int_{\dOm}
\frac1{|x-y|^{1/2}} \, d\h^1(y)\leq C(\Om) \|f\|^2_{L^2(\Om)}$. The claimed inequality follows by the Cauchy-Schwarz inequality.
}
\begin{align*}
|{\cal A}|&=\bigg|\int_0^h\int_0^h\int_{\Om} \int_{\Om} \frac{\nah(x)\, \nah(y)}{\sqrt{|x-y|^2+(x_3-y_3)^2}}\, d\bfx d\bfy\bigg|\\
&\leq h^2 \int_{\Om} \int_{\Om} \frac{|\nah(x)|\, |\nah(y)|}{|x-y|}\, dxdy\\
&\leq C h^2 \int_{\Om} |\nah|^2\, dx\leq Ch^2|\log \eps| \eeeh(\bm),
\end{align*}
and\footnote{Note that the terms in the integrand of $\cal B$ corresponding to the top and bottom boundary $\partial \Bfom_h$ will cancel after integration.}
  \begin{align*}|{\cal B}|&=\bigg|\int_{\Bfom_h} \int_{\partial \Bfom_h} \frac{\nah(x)\, (\bm\cdot \bfnu)(\bfy)}{|\bfx-\bfy|}\, d\bfx d\bfy\bigg|\\
  &=\bigg|\int_0^h\int_0^h \int_{\Om} \int_{\partial \Om} \frac{\nah(x)\, \bmn(y)}{\sqrt{|x-y|^2+(x_3-y_3)^2}}\, d\bfx d\bfy\bigg|\\
  & \leq h^2 \int_{\Om} \int_{\partial \Om} \frac{|\nah(x)|\, |\bmn(y)|}{|x-y|}\, dx dy\\
  &\leq Ch^2 \|\nah \|_{L^2(\Om)} \|\bmn\|_{L^2(\partial \Om)} \\
  &\leq Ch^2\eps^{1/2}|\log \eps| \eeeh(\bm),\end{align*}
 \noindent  while $${\cal C}=\int_{\partial \Bfom_h} \int_{\partial \Bfom_h} \frac{(\bm\cdot \bfnu)(\bfx)\, (\bm\cdot \bfnu)(\bfy)}{|\bfx-\bfy|}\, d\bfx d\bfy
  ={\cal C}_1+{\cal C}_2$$ 
  with\footnote{Note that the terms in the integrand of ${\cal C}$ corresponding to the mixing term $\bmt(\bfx) \bmn(y)$ as $\bfx$ covers the top and bottom boundary $\partial \Bfom_h$ will cancel after integration.} $${\cal C}_1=4\pi h \int_{\Om} \int_{\Om} {\bmt(x)\, \bmt(y)}{\Gamma_h(x-y)}\, dx dy  $$
  where $\Gamma_h(x)=\frac{1}{2\pi h} \bigg(\frac{1}{|x|}-\frac{1}{\sqrt{|x|^2+h^2}}\bigg)$, $x\in \R^2$ 
  and
  \begin{align*}
  {\cal C}_2=\int_0^h \int_0^h \int_{\partial \Om} \int_{\partial \Om} \frac{\bmn(x)\, \bmn(y)}{|\bfx-\bfy|}\, d\bfx d\bfy.\end{align*} 

\medskip

\noindent {\bf Estimate of ${\cal C}_1$}: The main novelty compared to the study in Kohn-Slastikov \cite{KohnSlastiko:2005a} is the
 following result, which replaces a limit calculation by a more quantitative estimate:
\be
\label{estimC1}
\frac{1}{\eta^2 |\log \eps|}\bigg|\frac{{\cal C}_1}{4\pi h }-\int_{\Om} \bmt^2(x)\, dx \bigg|\leq C \frac{h}{\eta^2} \big(\frac{\log \frac{\eta^2}{h}}{|\log \eps|} 
+ 1\big) \eeeh(\bm).
\ee
For that, since $\diam(\Omega)=1$ we can use that 
 $\Gamma_h(x)=\frac{h}{2\pi |x|^2} \rh(|x|)$ for $x\in B^2\subset \RR^2$  where 
$$\rh(r)=\frac{r}{(r+\sqrt{r^2+h^2})\sqrt{r^2+h^2}}  {\mathds 1}_{\{0\leq r\leq 1\}}(r), \quad r\geq 0.$$ First, note that $\{\rh\}_{h\downarrow 0}$ is a bounded sequence in $L^1(\RR^2)$, i.e.,
\be
\label{rho_boun}
\int_{\RR^2} \rh(|x|)\, dx\leq \pi.\ee Moreover, for every $R\in (0,1]$, one computes:
\begin{align*}
\int_{B_R(0)} \Gamma_h(x)\, dx&=h\int_0^R \frac{dr}{(r+\sqrt{r^2+h^2})\sqrt{r^2+h^2}}\\
&=\int_0^{R/h} \frac{ds}{(s+\sqrt{s^2+1})\sqrt{s^2+1}}=1-\left(\frac{R}{h}+\sqrt{1+\big(\frac{R}{h}\big)^2}\right)^{-1}\leq 1.
\end{align*} 
In particular, we get
\be
\label{esti2}
0\leq 1-\int_{B_R(0)} \Gamma_h(x)\, dx\leq \frac{h}{R} \quad \textrm{ for } \quad R\in (0,1].
\ee
Since $\Om$ is $C^{1,1}$, there exists $r_1:=r_1(\Omega)\in (0, 1=\diam(\Om))$ such that every point $x\in \Omega$ with $\dist(x, \partial \Om)<r_1$ has a unique orthogonal projection on the boundary $\partial \Om$, i.e., the crossing of two normal directions on $\partial \Om$ in the interior of $\Om$ happens at a distance larger than $r_1$ from the boundary.
For $R<r_1$ we denote by 
\be
\label{not_om}
\Om_R=\{x\in\Omega\, :\, \dist(x,\partial \Om)<R\} 
\ee
the region around the boundary $\partial \Om$ at a distance less than $R$.
Writing $2\bmt(x)\bmt(y)=\bmt(x)^2+\bmt(y)^2-(\bmt(x)-\bmt(y))^2$, we obtain that:
$${\cal C}_1=-{\cal E}_1+{\cal E}_2$$
with
\begin{align*}
{\cal E}_1&=h^2 \int_{\Om} \int_{\Om} \frac{(\bmt(x)-\bmt(y))^2}{|x-y|^2} {\rh(|x-y|)}\, dx dy\\
&\leq h^2 \int_{\RR^2} \int_{\RR^2} \int_0^1 \bigg|\nabla [T(\bmt)](x+s(y-x))\bigg|^2 {\rh(|x-y|)}\, dx dy ds\\
&\leq h^2 \int_{\RR^2} |\nabla [T(\bmt)](x)|^2 \, dx \int_{\RR^2} \rh(|y|)\, dy\\
&\stackrel{\eqref{rho_boun}}{\leq} C h^2 \int_{\Om}\left( |\nabla \bmt|^2{+\bmt^2} \right)\, dx
\leq Ch^2|\log \eps| \eeeh(\bm),
\end{align*}
where $T:H^1(\Om)\to H^1(\RR^2)$ is a linear continuous extension operator, $\bmt^2\leq 1-|\bar m_h|^2$, $\eta\leq 1$
and
$${\cal E}_2=4\pi h \int_{\Om} \int_{\Om} {\bmt^2(x)} {\Gamma_h(|x-y|)}\, dx dy.$$
It remains to estimate ${\cal E}_2$. 
As $\eta \to 0$, we may assume {in the regime \eqref{regime} that $2h\leq \eta^2\leq \frac{r_1}{2}$}.
By decomposing the domain $x\in \Om=\Om_{h}\cup (\Om_{\eta^2} \setminus 
\Om_{h})\cup (\Om \setminus 
\Om_{\eta^2})$ (with the notation \eqref{not_om}), since $\bmt^2\leq 1-|\bar m_h|^2\leq 1$, we compute:\footnote{We use that $1- \int_{\Om} {\Gamma_h(|x-y|)}\, dy\leq 1- \int_{B(0, \dist(x, \partial \Om))} {\Gamma_h(|z|)}\, dz$ if $x \in \Om_{\eta^2} \setminus 
\Om_{h}$, and $1- \int_{\Om} {\Gamma_h(|x-y|)}\, dy\leq 1- \int_{B(0, \eta^2)} {\Gamma_h(|z|)}\, dz$ if $x \in \Om \setminus 
\Om_{\eta^2}$. } 
\begin{align*}
&\bigg|\frac{{\cal E}_2}{4\pi h}-\int_{\Om} \bmt^2(x)\, dx \bigg|\stackrel{\eqref{esti2}}{=} \int_{\Om} \bmt^2(x)\left(1- \int_{\Om} {\Gamma_h(|x-y|)}\, dy \right) \, dx\\
&\stackrel{\eqref{esti2}}{\leq} \int_{\Om_{h}} 1  \, dx +\int_{\Om_{\eta^2} \setminus 
\Om_{h}}\, \frac{h}{\dist(x, \partial \Om)}dx+
\frac{h}{\eta^2}\int_{\Om\setminus 
\Om_{\eta^2}}(1- \bmp^2(x)) \, dx\\
&\leq C \left(h+h\int_h^{\eta^2} \frac{dr}{r}\right)+h |\log \eps| \eeeh(\bm)\leq C h \big( \log \frac{\eta^2}{h} + |\log \eps|\big)  \eeeh(\bm),
\end{align*} where $C=C(\partial \Om)>0$ depends only on the geometry of $\Om$
and we have used Lemma \ref{lem:twopiboundforaven} below yielding 
$\eeeh(\bm)\geq 2\pi-o(1)$ as $h\to 0$. 
 Thus, \eqref{estimC1} is proved.

\medskip

\noindent {\bf Estimate of ${\cal C}_2$}: We prove that
\be
\label{estimC2}
\frac{1}{|\log \eps|} \bigg| \frac{{\cal C}_2}{4\pi \eta^2 h}-\frac{1}{2\pi \eps} \int_{\partial \Om} \bmn^2\, d\h^1\bigg| { \ll} \frac{h}{\eta^2}
\eeeh(\bm).
\ee
This estimate is similar to  
 Lemma 4 in
Kohn-Slastikov \cite{KohnSlastiko:2005a}, 
but more delicate to prove in our regime. 
The idea here is to use a stronger estimate inspired by the work of Carbou \cite{Carbou:2001a}. 
For that,
we write
$$\frac{{\cal C}_2}{4\pi \eta^2 h |\log \eps|} =\frac{{\cal G}_1+{\cal G}_2}{4\pi}$$
with the rescaled quantities in $h$:
\begin{align*}
{\cal G}_1&=\frac{h}{ \eta^2|\log \eps|}  \int_{\partial \Om} \int_{\partial \Om} \bmn^2(x) K_h(x-y) dx dy
\end{align*}
where 
\be
\label{kh}
K_h(z)=\int_0^1 \int_0^1 \frac{1}{\sqrt{|z|^2+h^2(s-t)^2}}\, ds dt, \quad z\in \R^2
\ee
and
\begin{align*}
|{\cal G}_2|&=\frac{h}{ \eta^2|\log \eps|} \bigg| \int_0^1 \int_0^1 \int_{\partial \Om} \int_{\partial \Om} \frac{\bmn(x) (\bmn(x)- \bmn(y))}{\sqrt{|x-y|^2+h^2(s-t)^2}}\, dx dy ds dt\bigg|\\
&\leq \frac{h}{ \eta^2|\log \eps|}  \int_{\partial \Om} \int_{\partial \Om} |\bmn(x)| \frac{ |\bmn(x)- \bmn(y)|}{|x-y|}\, dx dy \\
&\leq \frac{Ch}{ \eta^2|\log \eps|} \|\bmp\cdot \nu\|_{L^{2}(\partial \Om)}
 \|\bmp\cdot \nu\|_{\dot{H}^{1/2}(\partial \Om)}\\
 & \leq  \frac{Ch}{ \eta^2|\log \eps|} \sqrt{\eps |\log \eps|\eeeh(\bm)}
 \|\bmp\cdot \nu\|_{\dot{H}^{1}(\Om)}
 \leq  \frac{Ch\eps^{1/2}}{ \eta^2}  
  \eeeh(\bm)
  {\ll \frac{h}{\eta^2}\eeeh(\bm)}.
\end{align*}
Above, we used the Cauchy-Schwarz inequality and we have extended $\nu$ as a Lipschitz 
vector field in $\overline{\Om}$ with $|\nu|\le 2$ so that Lemma \ref{lem:twopiboundforaven} below yields for $h>0$ small:
$$\|\bmn\|_{\dot{H}^{1}(\Om)}\leq C(\|\bmp\|_{\dot{H}^{1}(\Om)}+1)\lesssim \sqrt{ |\log \eps|\eeeh(\bm)}+ \sqrt{\eeeh(\bm)} \lesssim \sqrt{ |\log \eps|\eeeh(\bm)}.$$

It remains to estimate ${\cal G}_1$.
In fact, as $\eps=\frac{\eta^2}{h|\log h|}$, 
one has that
\begin{align*}
&\bigg| \frac{{\cal G}_1}{4\pi} -\frac{1}{2\pi \eps |\log \eps|} \int_{\partial \Om} \bmn^2\, d\h^1\bigg|\\
&\leq \frac{1}{4\pi \eps|\log \eps|} \int_{\partial \Om} \bmn^2  \,dx\, \bigg\|2-\frac{1}{|\log h|}\int_{\partial \Om} K_h(x-y) dy\bigg\|_{L^\infty(x\in \partial \Om)} \\
&\stackrel{Lemma\, \,  \ref{lem:estimG1}}{\leq} \frac{C}{|\log h|} \eeeh(\bm) = \frac{h\eps}{\eta^2} \eeeh(\bm)
\stackrel{\eqref{regime}}{\ll} \frac{h}{\eta^2} \eeeh(\bm). 
\end{align*}
Therefore, \eqref{estimC2} follows and the proof of Lemma~\ref{lem:claim} is finished.
\qed

\medskip

We have used above the  following estimate of $K_h$ in \eqref{kh} that permits to track its dependance in $h$, which is an improvement of a result 
of Carbou \cite[p. 1537]{Carbou:2001a}: 

\begin{lem}
\label{lem:estimG1}
Assume that $\partial\Om$ is a simply connected $C^{1,1}$ 
domain and let $K_h$ given by \eqref{kh}.
Then
\begin{equation}\label{eq:Khest}
\sup_{x\in\partial\Om} \left|\frac1{|\log h|} \int_{\partial\Om} K_h(x-y)dy-2\right|
\le \frac{C}{|\log h|} \quad {\textrm{for } h\in (0, \frac12].}
\end{equation}
\end{lem}
\begin{proof}{}
Note that by symmetry, for every $z\in \RR^2$,
\[
K_h(z)=\int_0^1 \int_0^1 \frac{1}{\sqrt{|z|^2+h^2(s-t)^2}}\, ds dt
= 2
\int_0^1 \int_0^t \frac{1}{\sqrt{|z|^2+h^2(s-t)^2}}\, ds dt.
\]
By a change of variable, the inner integral becomes:
\begin{align*}
\int_0^t \frac{1}{\sqrt{|z|^2+h^2(s-t)^2}}\, ds
=\frac1h\int_0^{ht} \frac{1}{\sqrt{|z|^2+w^2}}\, dw
=\frac1h \operatorname{arsinh}(\frac{ht}{|z|}), \quad z\neq 0,
\end{align*}
where $\operatorname{arsinh}(t)=\log (t+\sqrt{t^2+1})$ for $t\in \R$, so that integration by parts yields
\begin{align*}
K_h(z)
&=\frac2h\int_0^1 \operatorname{arsinh}(\frac{ht}{|z|})\, dt
=\frac 2 h f(\frac{|z|}{h}), \quad z\neq 0,
\end{align*}
where
$$f(t):=\operatorname{arsinh}\frac1t-\frac{1}{t+\sqrt{1+t^2}}>0, \quad t>0$$
is a positive and decreasing function on $(0, \infty)$. Moreover, we check that:\footnote{The second inequality follows for example by 
considering $g(s)=\frac1s f(\frac1s)$, since $g(s)\to\frac12$ and $g'(s)\to 0$ as $s\to 0$.}
$$\lim_{t\to 0}\frac{f(t)}{\log \frac 1 t}=1 \quad \textrm{and} \quad \bigg|tf(t)-\frac 12\bigg|\leq \frac C{t^2} \textrm{ as } t\to \infty,$$
so that
\be
\label{primitivul}
\bigg|\int_0^t f(s)\, ds-\frac{\log t}{2}\bigg|\leq C \quad \textrm{and} \quad  \bigg|\frac 1 t\int_0^t s f(s)\, ds\bigg|\leq C \quad \textrm{ as } t\to \infty.
\ee
For fixed $x\in\partial\Omega$, we set $z=x-y$ and integrate over $y\in\partial\Omega$.  We choose the arclength parameterization $\phi:[0,L)\to \partial \Omega$ such that $\phi(0)=x$ and $\phi$ is bijective ($\phi$ extends to a periodic $C^{1,1}$ function on $\RR$) with
$|\frac{d}{dt} \phi|= 1$ in $[0,L)$ (here, $L$ is the length of $\partial\Omega$). Since $\phi$ is $C^{1,1}$, Taylor's expansion implies for $\alpha=\frac 1 2 \|\frac{d^2}{dt^2}\phi\|_{L^\infty}{>0}$:
\be
\label{estk1}
|s|(1-\alpha|s|)\leq |\phi(t+s)-\phi(t)|\leq|s|(1+\alpha|s|), \quad \textrm{for every} \quad  t, s\in \R.\ee
Since $\phi$ is  continuous and injective {on every interval of length less than $L$}, then for every $\kappa\in (0, \frac L 2)$ there exists $\beta=\beta(\kappa)>0$ such that 
\be
\label{estk2}
L\ge |\phi(t+s)-\phi(t)|\geq \beta \quad \textrm{for every} \quad t\in \R, \, s\in (\kappa, L-\kappa).\ee
Fix some small $\kappa$ (more precisely, assume {$\kappa<\min \{ \frac{L}{100}, \frac1{2\alpha} \}$ and $0<2\kappa \alpha(1-\frac{\kappa \alpha}2)<\frac12$)}. Then 
\[
\int_{\partial\Omega} K_h(x-y) dy
= \frac2h \int_{-\kappa}^\kappa f(\frac{|\phi(s)-x)|}{h})ds
+ \frac2h \int_{\kappa}^{L-\kappa} f(\frac{|\phi(s)-x)|}{h})ds
=:I_1(h)+I_2(h).
\]
As $f$ is decreasing, we can estimate this from above and below using
estimates \eqref{estk1} and \eqref{estk2}. In particular, as $\phi(0)=x$,
\[
\frac4h \int_{0}^\kappa f(\frac{(1+\alpha s)s}{h})ds
\le I_1(h) 
\le \frac4h \int_{0}^\kappa f(\frac{(1-\alpha s)s}{h})ds
\]
and similarly,
\[
\frac{2(L-2\kappa)}{L} \bigg(\frac{L}{h}f(\frac{L}{h})\bigg)=\frac2h \int_{\kappa}^{L-\kappa} f(\frac{L}{h})ds
\le I_2(h)
\le
\frac2h \int_{\kappa}^{L-\kappa} f(\frac{\beta}{h})ds=\frac{2(L-2\kappa)}{\beta} \bigg(\frac{\beta}{h}f(\frac{\beta}{h})\bigg).
\]
As $tf(t)\to \frac12$ as $t\to \infty$, we obtain that
$$0\leq I_2(h)\leq C\quad \textrm{as } \, h\to 0.$$
It remains to prove that $I_1(h)\sim 2|\log h|$ as $h\to 0$. For that, 
note that by substitution, 
\[
\int_0^\kappa f(\frac{s(1\pm\alpha s)}{h})ds
=h\int_0^{\kappa(1\pm\alpha\kappa)/h} \frac{f(t)}{\sqrt{1\pm 4ht\alpha}}dt.
\]
\nd {\bf Lower bound for $I_1(h)$}. Observe that 
$\frac{1}{\sqrt{1+s}}\ge 1-\frac s 2$ for every $s>0$ (in particular for $s=4ht\alpha$). It follows that
\[
I_1(h)\ge 4\int_0^{\kappa(1+\alpha\kappa)/h} f(t)dt - 
8h\alpha\int_0^{\kappa(1+\alpha\kappa)/h} tf(t)dt\stackrel{\eqref{primitivul}}\geq 2|\log h| -C.
\]
\nd {\bf Upper bound for $I_1(h)$}.
We similarly use $\frac{1}{\sqrt{1-s}}\leq 1+2s$ for every $s\in[0, \frac 1 2]$ (in particular for $s=4ht\alpha\in [0,2\kappa \alpha(1-\kappa \alpha)]\subset [0, \frac 1 2]$ by our choice of some small fixed $\kappa$ and $h<1/2$).
It follows that
\[
I_1(h)\le 4\int_0^{\kappa(1-\alpha\kappa)/h} f(t)dt + 
32h\alpha\int_0^{\kappa(1-\alpha\kappa)/h} tf(t)dt\stackrel{\eqref{primitivul}}\leq 2|\log h| +C.
\]
Thus, \eqref{eq:Khest} follows. \qed
\end{proof}

\proof{ of Theorem \ref{thm:lowerb}}
We may assume $ \sup_h E_h(\bfm_h)\le K<\infty$. From \eqref{eq:eandbare} of Lemma~\ref{lem:l1}, we see that
$$E_h(\bfm_h)\ge \eeeh(\bm)-\left(\eeeh(\bm)+\sqrt\frac{{K}}{|\log\eps|}\right) \lambda(h),$$
where $\lambda(h)$ is given in \eqref{ah}.
From \eqref{regimelogs} 
 we see that  that $ \lambda(h)=o(1)$ as $\eps\to 0$,
and we can conclude that we must have
\[
\limsup_{h\to 0} \eeeh(\bm) \le K,
\]
so we obtain the bound
\[
E_h(\bfm_h)\ge \eeeh(\bm) - \lambda(h) \left(K+1+{ 2}\sqrt{\frac{K}{|\log\eps|}}\right)=\eeeh(\bm)-o(1) \quad \textrm{as} \, \, h\to 0.
\]
Furthermore, in the regime \eqref{regime2}, 
\[
 \lambda(h)  =\frac { \log|\log h|}{\eps |\log\eps| |\log h|}  \ll \frac{1}{|\log\eps|},
\]
and hence we obtain
\[
E_h(\bfm_h)\ge \eeeh(\bm) -o(\frac1{|\log\eps|})  \quad \textrm{as} \, \, h\to 0.
\]
If the  magnetizations $\bfm_h$ are  invariant in $x_3$-direction (when $\bfm_h$ coincides with the average $\bm$), 
then clearly
\[
\frac1h \int_{\Bfom_h} | \bfa \bfm_h|^2 \, d\bfx = \int_\Om |\nabla \bm|^2 \, dx,
\]
and since  $U_h=\bar U_h$ in $\R^3$, the  stray field term in $E_h(\bfm_h)$ and the penalty terms in  $\eeeh(\bm)$ are close to each other by \eqref{claim}, so 
the asymptotic inequalities become asymptotic equalities as claimed.

\qed

\section{Proof of Theorem~\ref{thm:t1}}\label{sec:sec3}
In this section we prove Theorem~\ref{thm:t1} as a consequence of the estimates in Theorem~\ref{thm:lowerb} and the results we obtained in \cite{IK_jac} 
for a functional related to $\eeeh$ in \eqref{eq:defofebar}. More precisely, in \cite{IK_jac}, we studied   
the following energy functional for $u\in H^1(\Om;\RR^2)$:
\be
\label{eq:introenepseta}
\eee(u) = \int_{\Omega} |\nabla u|^2 \, dx + \frac1{\eta^2}\int_\Omega (1-|u|^2)^2 \, dx 
+ \frac{1}{2\pi\eps} \int_{\dOm} (u\cdot \nu)^2 \, d\h^1, \quad \eps, \eta >0.
\ee
In our context  $\eps=\eps(h)$ and $\eta=\eta(h)$, note that for $u\in H^1(\Om;\RR^2)$ with $|u|\le 1$  in $\Om$,
\be\label{eq:eeeheee}
\eeeh(u)\ge\frac1{|\log\eps|} \eee(u)
\ee
because $(1-|u|^2)^2\le (1-|u|^2)$ as $|u|\le 1$. We always use for $u\in H^1(\Om;\RR^2)$ the identification  $u\equiv (u,0)\in H^1(\Om;\RR^3)$ as $\eeeh$ is defined for $\bar B^3$-valued  maps.
Moreover, if  $u\in H^1(\Om; \Ss^1)$, then \eqref{eq:eeeheee} becomes equality:
\be\label{eq:eeeheee2}
\eeeh(u)=\frac1{|\log\eps|} \eee(u).
\ee
We recall here the $\Gamma$-convergence result that we established in our previous paper 
\cite[Theorems~1.2, 1.4 and 1.5]{IK_jac} that is essential in the sequel:

\begin{thm}[\cite{IK_jac}]
\label{thm:gammacforee}
Let $\Omega\subset \R^2$ be a bounded, simply connected $C^{1,1}$ domain, 
$\eps\to 0$ and $\eta=\eta(\eps)\to 0$ be sequences / families satisfying $|\log\eps|\ll |\log\eta|$. 
Assume $u_\eps\in H^1(\Om;\R^2)$ satisfy
\[
\limsup_{\eps\to 0} \frac{1}{|\log\eps|}E_{\eps,\eta}(u_\eps) <\infty.
\]
\begin{itemize}
\item[i)] {\textbf{Compactness of global Jacobians and $L^p(\dOm)$-compactness of $u_\eps\big|_{\partial \Om}$}. 
For a subsequence,
 the global Jacobians $\jaco(u_\eps)$ 
 converge  to 
a measure $J$ on the closure $\overline{\O}$ in the sense
 of \eqref{conv_lip},  
$J$ is supported on $\dOm$ and has the form \eqref{newlab}
for $N$ distinct boundary vortices $a_j \in \partial\O$ carrying the {non-zero} {multiplicities}
$d_j\in\ZZ\setminus \{0\}$. Moreover, for a subsequence, the trace $u_\eps\big|_{\partial \Om}$ converges as $\eps\to 0$ in $L^p(\dOm)$ (for every $p\in [1, \infty)$) to $e^{i\phi}\in BV(\dOm; \Ss^1)$ for a lifting $\phi$ of the tangent field $\pm \tau$ on $\dOm$ determined (up to a constant in $\pi \ZZ$) by 
$\partial_\tau \phi =-J$ on $\dOm$.
}

\item[ii)] \textbf{Energy lower bound at the first order}. If $(u_\eps)$ satisfies the convergence 
assumption on the Jacobians  as the sequence / family $\eps\to 0$ as in $i)$, then\footnote{Recall that $J+\ka  {\h^1\llcorner \dOm} = \pi \sum_{j=1}^N d_j \delta_{a_j}$.}
\begin{equation}\label{eq:fogc1}
\liminf_{\eps\to 0} \frac1{|\log\eps|}E_{\eps,\eta}(u_\eps) \ge\pi\sum_{j=1}^N |d_j|=\big|J+\ka  {{\h^1\llcorner \dOm}}\big|(\dOm).
\end{equation}
\end{itemize}
If we additionally  assume the following sharper bound:
\be
\label{eq:sharpenbd}
\limsup_{\eps\to 0} \bigl(E_{\eps,\eta}(u_\eps) - 
 \pi\sum_{j=1}^N |d_j||\log\eps| \bigr) <\infty,
\ee
then the following results hold:
\begin{itemize}
\item[iii)] \textbf{Single multiplicity and second order lower bound}.
The multiplicities satisfy $d_j=\pm 1$ for $1\leq j\leq N$, { so $\sum_{j=1}^N |d_j|=N$}  and 
{there holds the finer energy bound}
\be\label{eq:fogc2}
\liminf_{\eps\to 0} \bigl(E_{\eps,\eta}(u_\eps) - \pi {N} |\log\eps| \bigr) 
\ge W_\Omega({\{(a_j,d_j)\}}) + {\gamma_0}{N},
\ee
with $\gamma_0=\pi\log\frac{ e}{4\pi}$ a universal constant and $W_\Omega$ the renormalised
energy defined in \eqref{eq:renW}. 
\item[iv)] \textbf{Penalty bound.} 
The penalty terms {are of order $O(1)$, i.e., }
\begin{equation}\label{eq:thpenub}
\limsup_{\eps\to 0}\left( \frac1{\eta^2} \int_\Omega (1-|u_\eps|^2)^2\, dx + \frac1{2\pi \eps} \int_{\partial\Omega} 
(u_\eps \cdot \nu)^2 \, d\h^1 \right)<\infty.
\end{equation}
\item[v)] \textbf{Local energy lower bound.}
There are $\rho_0>0$, $\eps_0>0$ and $C>0$ such that 
the energy of $u_\eps$ near the singularities satisfies for all the {$\eps<\eps_0$ in the sequence / family } and $\rho<\rho_0$: 
\begin{equation}\label{eq:l2lob}
\left( \int_{\Omega\cap \bigcup_{j=1}^N B_\rho(a_j)} |\nabla u_\eps|^2 \, dx -\pi {N} \log \frac\rho\eps\right) > -C.
\end{equation}
\item[vi)] \textbf{$L^p(\Om)$-compactness of maps $u_\eps$.}
For any $q\in[1,2)$, 
the sequence /family $(u_\eps)_\eps$ is uniformly bounded in $W^{1,q}(\Omega;\R^2)$. Moreover, for a subsequence, $u_\eps$ converges as $\eps\to 0$ strongly in $L^p(\Omega;\R^2)$ for any $p\in [1, \infty)$ to 
$e^{ i\hat\phi}$, where  $\hat\phi\in W^{1,q}(\Omega)$ is an extension (not necessarily harmonic) to $\Omega$ of the lifting $\phi\in BV(\dOm)$ determined in point i).  
\end{itemize}
Finally, we have a matching \textbf{upper bound construction}:
\begin{itemize}
\item[vii)] Given any $N$ distinct points $a_j\in \partial \O$ with their multiplicity $d_j\in \ZZ\setminus \{0\}$
satisfying the constraint $\sum_{j=1}^N d_j=2$,
we can construct for every $\eps\in (0, \frac12)$, $u_\eps\in H^1(\Omega; \Ss^1)$ such that the global Jacobians $\jaco(u_\eps)$ converge to $J=-\ka \h^1\llcorner \dOm+\pi \sum_{j=1}^N d_j \delta_{a_j}$ as in \eqref{conv_lip}.
Furthermore, $u_\eps$ converge strongly  to  a canonical harmonic map $m_*$ associated to $\{(a_j,d_j)\}$ in $L^p(\Om)$ and $L^p(\dOm)$ for all $p\in[1,\infty)$.
 and the energies satisfy
\[
\lim_{\eps\to 0} \frac{1}{|\log\eps|}E_{\eps,\eta}(u_\eps) = \pi \sum_{j=1}^N |d_j|.
\]
If furthermore $|d_j|=1$ for all $j=1,\dots,N$, then $u_\eps$ can be chosen such that
\[
\lim_{\eps\to 0}  (E_{\eps,\eta}(u_\eps)-\pi N|\log\eps|) = W_\Omega(\{(a_j, d_j)\})+ N\gamma_0.
\]
\end{itemize}

\end{thm}

Using Theorem \ref{thm:gammacforee}, we first show a uniform lower bound for $\eeeh$ that is required in the proof of Lemma~\ref{lem:l1} (which we used for proving Theorem~\ref{thm:lowerb}). The assumption that $\Omega$ is simply connected (or at least not topologically an annulus, compare Remark \ref{rem:multiconn}) is very important. As the result is only about $\eeeh$, our reasoning to establish Theorem~\ref{thm:lowerb} is not circular.
\begin{lem}\label{lem:twopiboundforaven}
Let $\bfm_h\in H^1(\Omega;\R^3)$ with $|\bfm_h|\le 1$ in $\Omega$. Then 
\[
\liminf_{h\to 0} \eeeh(\bfm_h) \ge 2\pi.
\]
\end{lem}
\begin{proof}{}
We denote $\bfm_h=(m_h, m_{3,h})$. As $\eeeh(\bfm_h)\geq \eeeh(m_h,0)$, we may assume $\bfm_h = (m_h,0)$ and $\limsup_{h\to 0} \eeeh(\bfm_h) \le C$.
Writing $\eps=\eps(h)$, $\eta=\eta(h)$ and $u_\eps:=m_h\in H^1(\Omega;\R^2)$, we have from \eqref{eq:eeeheee} that 
$\eeeh(\bfm_h) \ge \frac1{|\log\eps|} \eee(u_\eps)$. {We furthermore restrict ourselves to a subsequence $\eps\to 0$ such that $\eee(u_\eps)/|\log \eps|\to \liminf_{\eps\to 0} \eee(u_\eps)/|\log \eps|$.} Hence we can apply Theorem~\ref{thm:gammacforee} and obtain for a further subsequence that 
the global Jacobians $\jaco(u_\eps)$ converge
  to a measure $J$ supported on the boundary $\partial\O$ 
of the form $J=-\ka \h^1\llcorner \dOm+\pi \sum_{j=1}^N d_j \delta_{a_j}$, $a_j \in \partial\O$, 
$d_j\in\ZZ\setminus \{0\}$, $\sum_{j=1}^N d_j=2$; moreover, Theorem~\ref{thm:gammacforee} yields 
\[
\liminf_{\eps\to 0} \frac1{|\log\eps|}E_{\eps,\eta}(u_\eps) \ge\pi\sum_{j=1}^N |d_j| \ge \pi \bigg|\sum_{j=1}^N d_j\bigg| = 2\pi,
\]
which proves the claim. \qed
\end{proof}

\bigskip

We are now in the position to prove  Theorem~\ref{thm:t1}.

\bigskip

\begin{proof}{ of Theorem~\ref{thm:t1}}
We start by proving the compactness statement in $(i)$.
For that, we assume $\bfm_h:\Bfom_h\to \Ss^2$ is a sequence of 
magnetizations such that $\sup_h E_h(\bfm_h)\le C$ and we set
$\bm =(\bar m_h, \bar m_{h,3})$ the average magnetization in \eqref{eq:averageofm}.
By Theorem~\ref{thm:lowerb} and \eqref{eq:eeeheee}, we have 
\[C\geq \limsup_{h\to 0} \eeeh(\bm)\geq \limsup_{\eps\to 0} \frac1{|\log\eps|}\eee(\bar m_h).\]
Note that \eqref{regime} implies $|\log \eps|\ll \frac1\eps\ll |\log \eta|$. Thus, the compactness
statement of 
Theorem~\ref{thm:gammacforee} applies to $u_\eps:=\bar m_h$ yielding
the relative compactness of $\jaco(\bar m_h)$ as well as the $L^p(\dOm)$-relative compactness of the traces $\bar m_h\bigg|_{\dOm}$ for every $p\in [1, \infty)$; as every limit of this sequence takes values in $\Ss^1$ on $\dOm$ and
$\bar m_{h,3}^2\leq 1-|\bar m_h|^2$, it yields $\bar m_{h,3}\to 0$ in $L^2(\dOm)$. As $|\bar m_{h,3}|\leq 1$, we conclude for a subsequence,  $\bar m_{h,3}\to 0$ in $L^p(\dOm)$  for every $p\in [1, \infty)$ which proves $(i)$.

From \eqref{eq:fogc1} and Theorem \ref{thm:lowerb}, we establish $(ii)$:
$$\liminf_{h\to 0} E_h(\bfm_h) \geq \liminf_{h\to 0} \eeeh(\bm)\geq  \liminf_{\eps\to 0} \frac1{|\log\eps|}\eee(\bar m_h)\geq \pi\sum_{j=1}^N |d_j|.$$
 Moreover, in the regime \eqref{regime2}, the assumption
\(
\limsup_{h\to 0} |\log \eps| (E_h(\bfm_h) - \pi\sum_{j=1}^N |d_j|) \le C, 
\) together with the improved estimate \eqref{eq:improvedest} of Theorem~\ref{thm:lowerb} yields
\eqref{eq:sharpenbd}; thus, by Theorem~\ref{thm:gammacforee}, we have $|d_j|=1$ for every $1\leq j\leq N$ and together with \eqref{eq:fogc2}, the statement $(iii)$ also follows. 

To prove $(iv)$, note that \eqref{eq:l2lob} together with Jensen's inequality imply for $\eps<\eps_0$ and $\rho<\rho_0$:
\[
\frac1h \int_{\left(\Om\cap\bigcup_{j=1}^N B_\rho(a_j)\right)\times(0,h)}  |\bfa \bfm_h|^2 \, d\bfx \ge  \int_{\Om\cap\bigcup_{j=1}^N B_\rho(a_j)}  |\nabla \bar m_h|^2 \, dx \ge \pi N \log\frac\rho\eps-C.
\]
Combined with the upper bound for the energy
\[
\frac1h \int_{\Bfom_h}  |\bfa \bfm_h|^2 \, d\bfx\leq |\log \eps| E_h(\bfm_h) \le \pi N |\log\eps| + C,
\]
we obtain for every $\rho<\rho_0$:
\[
\frac1h \int_{\left(\Om\setminus\bigcup_{j=1}^N B_\rho(a_j)\right)\times(0,h)}  |\bfa \bfm_h|^2 \, d\bfx \le C(1+\log\frac1\rho).
\]
We now use the argument of Struwe \cite{Struwe:1994a} to obtain $L^q$ bounds for $\bfa \bfm_h$ for every $q\in [1,2)$ (see also \cite[Lemma 4.17]{IK_jac}). Set 
$\bfom_s=(\Om\setminus \bigcup_{j=1}^N B_s(a_j))\times(0,h)$. Then for all $s$, $|\bfom_{s}\setminus\bfom_{2s}|\le  Cs^2h$.
With  $s_k=2^{1-k}\rho_0$, $k=1,2,\dots$, we find, adjusting the constant, that 
\be\label{eq:zhege}
\int_{\bfom_{s_k}} |\bfa \bfm_h|^2 \, d\bfx \le C h k.
\ee
We compute using H\"older's inequality and \eqref{eq:zhege}:
\begin{align*}
\int_{\Bfom_h} |\bfa \bfm_h|^q \, d\bfx &= \int_{\bfom_{s_1}}  |\bfa \bfm_h|^q \, d\bfx + \sum_{k=1}^\infty \int_{\bfom_{s_{k+1}}\setminus \bfom_{s_k}}  |\bfa \bfm_h|^q \, d\bfx \\
&\le Ch + \sum_{k=1}^\infty \left(\int_{\bfom_{s_{k+1}}\setminus \bfom_{s_k}}  |\bfa \bfm_h|^2 \, d\bfx\right)^{\frac q2} 
\left|{\bfom_{s_{k+1}}\setminus \bfom_{s_k}}\right|^{1-\frac q2} \\
&\le Ch + C\sum_{k=1}^\infty (kh)^{\frac q 2} \left(2^{-2k} h\right)^{1-\frac q 2}
\le Ch + Ch \sum_{k=1}^\infty k^{\frac q2}2^{-k(2-q)}\le C(q)h,
\end{align*}
where the above infinite sum is convergent since we assumed $1\leq q<2$. 
By 
rescaling $\tilde \bfm_h(x,x_3)=\bfm_h(x,hx_3)$ with $\tilde \bfm_h=(\tilde m_h, \tilde m_{h,3}):\Bfom_1\to \Ss^2$, we have for $h<1$ and $1\leq q<2$:
\be
\label{555}
\|\bfa \tilde\bfm_h\|^q_{L^q(\Bfom_1)}\le \int_{\Bfom_1} |\nabla \tilde \bfm_h|^q+\frac{|\partial_{x_3} \tilde \bfm_h|^q}{h^q}\, d\bfx=\frac1h \int_{\Bfom_h} |\bfa \bfm_h|^q \, d\bfx \leq  C(q).
\ee 
By compact embedding of $W^{1,q}(\Bfom_1) \subset L^{2}(\Bfom_1)$ for $q\in (1,2)$, we obtain after extraction the
 strong convergence $\tilde\bfm_h\to \tilde\bfm$ in $L^2(\Bfom_1)$ to a limit $\tilde\bfm\in W^{1,q}(\Bfom_1)$. For a further subsequence, we can assume pointwise convergence almost 
 everywhere, so we have $|\tilde\bfm|=1$ almost everywhere as $|\tilde \bfm_h|=1$. Moreover, we deduce
 the strong convergence $\tilde\bfm_h\to \tilde\bfm$  in $L^p(\Bfom_1)$ for all $1\le p<\infty$. 
 By \eqref{555}, we also get $\|\partial_{x_3} \tilde\bfm_h\|_{L^q}\to 0$ for every $q\in [1,2)$,  so $\partial_{x_3}
\tilde  \bfm=0$ in $\Bfom_1$. Hence $\tilde\bfm$ is equal to its $x_3$-average, i.e., $\tilde \bfm=\tilde \bfm(x):\Omega\to \Ss^2$ (which obviously is the $L^p(\Omega)$-limit of the $x_3$-average of $\tilde\bfm_h$ since $\tilde\bfm_h\to \tilde\bfm$  in $L^p(\Bfom_1)$ for all $1\le p<\infty$). Note that the $x_3$-average of $\tilde\bfm_h$ coincides with $\bm$ in $\Omega$, so the $L^p(\Omega)$-limit of $\bm$ is $\tilde \bfm$. By point $vi)$ in Theorem \ref{thm:gammacforee} on $L^p(\Omega)$-limiting behaviour of $u_\eps=\bar m_h$, we then deduce that the limit $\tilde\bfm$ has the form $(\tilde m,0)$ with $\tilde m:\Omega\to \Ss^1$. Moreover, Jensen's inequality applied in \eqref{555} yields $\nabla\bar m_h$ is bounded\footnote{This is also a consequence of point $vi)$ in Theorem \ref{thm:gammacforee}.} in $L^q(\Omega)$ for every $q\in [1,2)$. Then
the strong $L^{3}(\Omega)$-convergence of $\bar m_h$ and weak $L^{3/2}(\Omega)$-convergence of $\nabla\bar m_h$  imply that 
$$-\left< \jaco(\bar m_h), \zeta\right>=\int_{\Omega} \bar m_h \times \nabla \bar m_h \cdot \nabla^\perp \zeta\, dx
\to \int_{\Omega} \tilde m \times \nabla \tilde m \cdot \nabla^\perp \zeta\, dx=-\left< \jaco(\tilde m), \zeta\right>, $$
for every $\zeta\in W^{1,\infty}(\Om)$. By point $(i)$, we conclude that $\jaco(\tilde m)=J$ as claimed.
\qed
\end{proof}
\bigskip

We now show the corresponding upper bound:
\begin{proof}{ of Theorem~\ref{thm:t1UB}}
For $j=1,\dots, N$ let
 $a_j\in \dOm$ be distinct points and $d_j\in\ZZ\setminus \{0\}$ such that $\sum_{j=1}^N d_j=2$.
Let $u_\eps:\Om\to \Ss^1$ be chosen as in part $vii)$ of Theorem~\ref{thm:gammacforee}. 
Set $\bfM_h(x,x_3)=(u_\eps(x),0)\in \Ss^1\times \{0\}$ for $0<x_3<h$ and $x\in \Omega$. Clearly $\bfM_h$ satisfies (i).
As $\bfM_h$ is $x_3$-independent, then $\bfM_h$ coincides with its average $\bar \bfM_h$ in \eqref{eq:averageofm} with the in-plane components $M_h=u_\eps$ in $\Omega$. As $|u_\eps|=1$ in $\Omega$ and $\bar\bfM_h$ is in-plane in $\Om$, \eqref{eq:eeeheee2} yields $\eeeh(\bar\bfM_h)=\frac{1}{|\log\eps|} \eee(u_\eps)$. By the choice of $u_\eps$, $\jaco(M_h)=\jaco(u_\eps)\to J$ in the sense of \eqref{conv_lip}  and we also have the desired convergence of $M_h$ to $m_*$. 
By Theorem \ref{thm:lowerb} and the $x_3$-independence of $\bfM_h$, the choice of $u_\eps$ implies in the regime \eqref{regime}:
\[
E_h(\bfM_h)=\eeeh(\bar\bfM_h)+o(1)= \frac1{|\log\eps|}\eee(u_\eps)+ o(1)\to \pi \sum_{j=1}^N |d_j|\quad \textrm{as }\, h\to 0. \]
This is (ii). To show (iii), we just note that if additionally $|d_j|=1$ for all $j$, the above argument yields in the regime \eqref{regime2} via part $vii)$ in Theorem~\ref{thm:gammacforee}:
\[
\lim_{h\to 0} |\log\eps| (E_h(\bfM_h)-\pi N) =\lim_{h\to 0}  |\log\eps| \bigg(\eeeh(\bar\bfM_h)+o(\frac1{|\log \eps|})-\pi N\bigg)=W_\Omega(\{(a_j,d_j)\})+ N\gamma_0.
\]
\qed
\end{proof}

\section{Canonical harmonic maps and the renormalized energy}
\label{sec:ren}
In this section we compute the renormalized energy defined in \eqref{eq:renW} and prove the existence of minimisers in the situation of two boundary vortices of multiplicity $1$. 
First, we compute the renormalized energy in terms of a solution to a Neumann problem (see \eqref{eq:VisNeum}) which is similar to the method of Bethuel-Brezis-H\'elein \cite{BethuelBrezisHelein:1994a}
(compare also \cite[Proposition 7.1]{Kurzke:2006b}). Second, we prove Theorem \ref{pro:ree}, i.e.,  an exact formula in the situation of a disk domain that we transfer afterwards to a general bounded $C^{1,1}$ simply connected domain via a conformal map.

We start by proving the following formula of the renormalized energy proving in particular that the limit in \eqref{eq:renW} exists:
\begin{pro} 
\label{pro:ren_BBH}
Let $\Omega\subset \RR^2$ be a  simply connected $C^{1,1}$ domain with outer unit normal field $\nu$ and let  $\ka$ denote the curvature of the boundary $\dOm$. 
We consider $N\geq 2$ distinct points $a_j\in \dOm$ carrying the multiplicities $d_j\in \{\pm 1\}$ with $\sum_{j=1}^N d_j=2$. Then the limit in \eqref{eq:renW} exists and the renormalized energy satisfies
\be
\label{ren_1}
W_\O(\{(a_j,d_j)\}) = -\pi \sum_{k\neq j} d_k d_j \log |a_k-a_j| 
-\int_{\partial\O} \psi \ka  \,d\h^1 + \pi \sum_{j=1}^N d_j R(a_j),
\ee
where $\psi$ denotes the unique (up to an additive constant) solution  in $W^{1,q}(\Om)$ for every $q\in [1,2)$ of the inhomogeneous Neumann problem
\be
  \label{eq:VisNeum}
\begin{cases} 
 \Delta \psi &= 0 \qquad\text{in $\Omega$} \\
  \frac{\partial \psi}{\partial\nu} & = -\ka+\pi\sum_{j=1}^N d_j \delta_{a_j} \quad\text{on $\partial\Om$},
\end{cases}
\ee
and $R$ is the harmonic function given by
$R(z)=\psi(z)+\sum_{j=1}^N d_j \log|z-a_j|$.  It satisfies $R\in C^{0, \alpha}(\overline{\Omega})$ for every $\alpha\in (0,1)$ and  $R\in W^{s,p}(\Om)$ for every $p\in [1,\infty)$ and $s\in [1, 1+\frac1p)$.

\end{pro}

\begin{rem}
\label{rem4}
 While  $\psi$ is determined only up to an additive constant (idem for $R$), the above formula \eqref{ren_1} is independent of that constant due to the Gauss-Bonnet formula and the constraint $\sum d_j=2$. 
\end{rem}

\begin{rem}
In Definition~\ref{defi:renen}, we asked for  $e^{\ci\phi} \cdot \nu=0$. We can replace the unit vector field $\nu$ (of degree $1$) that we want $e^{\ci\phi}$
to be perpendicular to by a more general unit vector field $V$, compare \cite[Remark 1.2]{IK_jac} and \cite[Proposition 7.1]{Kurzke:2006b}.
\end{rem}

\begin{proof}{} 
First, we show that \eqref{eq:VisNeum} admits a unique solution (up to an additive constant) in $W^{1,q}(\Om)$ for every $q\in [1,2)$. Indeed, the existence result 
is a direct consequence of the fact that any conjugate harmonic function $\psi$ of $\phi_*$ introduced in 
Definition~\ref{defi:can_map} satisfies \eqref{eq:VisNeum}. Such a function $\phi_*$ exists since $-\int_{\partial\O} \ka d\h^1+\pi\sum_{j=1}^N d_j =0$ (by the Gauss-Bonnet theorem and the constraint $\sum_{j=1}^N d_j=2$); moreover, $\phi_*\in W^{s+\frac1p, p}(\Omega)$ for every $s\in (0,1)$ and $p\in (1, \frac1s)$ (because of the trace theorem and the fact that $\phi_*\big|_{\dOm}\in BV\cap L^\infty(\dOm)\subset W^{s, p}(\Omega)$ for every $s\in (0,1)$ and $p\in [1, \frac1s)$ by Sobolev embedding), in particular, $\phi_*\in W^{1,q}(\Om)$ for every $q\in [1,2)$; thus, a conjugate harmonic function $\psi$ of $\phi_*$ (i.e., $\phi_*+\ci \psi$ is holomorphic in $\Omega$) satisfies \eqref{eq:VisNeum} (via the Cauchy-Riemann equations)
and $\psi\in W^{1,q}(\Om)$ for every $q\in [1,2)$. The uniqueness (up to a constant) of $\psi \in W^{1,q}(\Om)$ for $q\in [1,2)$ is proved as follows\footnote{A second proof can be given using the uniqueness of the conjugate function $\phi$ (up to an additive function) explained in Remark \ref{rem-uniq}.}: if $\tilde \psi\in W^{1,q}$ is another solution of \eqref{eq:VisNeum}, then the difference $\hat \psi=\psi-\tilde \psi\in W^{1,q}(\Om)$ for $q\in [1,2)$ satisfies $\int_\Omega \nabla \hat \psi \cdot \nabla \zeta\, dx=0$ for every $\zeta \in W^{1,p}(\Om)$ for $p>2$. Up to subtracting a constant, we may assume that $\int_\Om \hat \psi\, dx=0$; letting $\zeta\in W^{3,q}(\Om)$ be a solution of $\Delta \zeta=\hat \psi$ in $\Om$ and $\frac{\partial \zeta}{\partial\nu}=0$ on $\dOm$, integration by parts yields 
$$\int_\Om \hat \psi^2\, dx=\int_\Om \hat \psi \Delta \zeta\, dx=-\int_\Om \nabla \hat \psi \cdot \nabla \zeta\, dx+\int_{\dOm} \hat \psi \frac{\partial \zeta}{\partial\nu}\, d\h^1=0,$$ 
i.e., $\hat \psi=0$ in $\Om$.

We now study the function $R(z)=\psi(z)+\sum_{j=1}^N d_j \log|z-a_j|$ defined for $z\in \Om$. The above properties of $\psi$ yield $R$ is harmonic in $\Om$ and $R\in W^{1,q}(\Om)$ for every $q\in [1,2)$. In fact, $R$ has better regularity. Indeed, note that $\frac{\partial}{\partial\nu} \log|z-a_j|=-\pi \delta_{a_j}+Q_j(z)$ for $z\in \dOm$ where $Q_j\in L^\infty(\dOm)$.\footnote{For $\dOm$, we use the counterclockwise arc-length parametrization $\gamma\in C^{1,1}$ with $|\gamma'|=1$ and $\gamma(0)=a_j$. Writing $\tau(t)=\gamma'(t)$ and $\nu(t)=-\gamma'(t)^\perp$, then 
$\frac{\partial \log|z-a_j|}{\partial\nu}\big|_{z=\gamma(t)}=\frac{(\gamma(t)-a_j)\cdot \nu(t)}{|\gamma(t)-a_j|^2}=:Q_j(\gamma(t))$ for every $t\neq 0$. This is a bounded function because $\gamma(t)-a_j=t\big(\tau(t)+O(t)\big)$ and $(\gamma(t)-a_j)\cdot \nu(t)=\int_0^t (\gamma(s)-a_j)\cdot \nu'(s)\, ds=\pm\int_0^t s (1+O(s)) \ka(s)\, ds=O(t^2).$}  Thus, by \eqref{eq:VisNeum}, we compute $\frac{\partial}{\partial\nu} R=-\ka+\sum_{j=1}^N d_j Q_j\in L^\infty(\dOm)$. Standard elliptic theory\footnote{One can deduce this regularity by using a conjugate harmonic function for $R$.} implies $R\in W^{s,p}(\Om)$ for every $p\in [1,\infty)$ and $s\in [1, 1+\frac1p)$, in particular, $R\in H^1(\Om)$ and $R\in C^{0, \alpha}(\bar \Om)$ for every $\alpha\in (0,1)$ by Sobolev embedding.

Finally, we prove \eqref{ren_1}. On $\Omega_\rho=\Omega\setminus \bigcup_{j=1}^N B_\rho(a_j)$ with the outer unit normal vector $\nu$, we now calculate
\[
\int_{\Omega_\rho} |\nabla \phi_*|^2 dx = \int_{\Omega_\rho} |\nabla \psi|^2 dx = \int_{\partial\Omega_\rho} \psi 
\frac{\partial\psi}{\partial\nu} d\h^1.
\]
The final integral can be split into 
\[
\int_{\partial\Omega_\rho} \psi 
\frac{\partial\psi}{\partial\nu} d\h^1
= \int_{\partial\Omega_\rho\cap \partial\Omega} \psi \frac{\partial\psi}{\partial\nu} d\h^1 -
\sum_{j=1}^N \int_{\partial B_\rho(a_j)\cap\Omega} \psi \frac{\partial\psi}{\partial\rho} d\h^1,
\]
with $\frac{\partial}{\partial\rho}$ denoting the radial derivative at $\partial B_\rho(a_j)$. 
The first term in the above RHS converges as $\rho\to 0$:
$$\int_{\partial\Omega_\rho\cap \partial\Omega} \psi \frac{\partial\psi}{\partial\nu} d\h^1=-\int_{\partial\Omega_\rho\cap \partial\Omega} \psi \ka d\h^1\to -\int_{\partial\O} \psi \ka\, d\h^1,$$ because $\ka \in L^\infty(\dOm)$ and $\psi\in L^1(\dOm)$ (by the trace theorem for $\psi\in W^{1,1}(\Om)$). 
For the second term, we observe that on $\partial B_\rho(a_j)\cap\Omega$, $\frac{\partial\psi}{\partial\rho}={ -}\frac{d_j}{\rho}+\frac{\partial R}{\partial\rho}+S_j(z)$, where $S_j$  is smooth in $\overline{B_\rho(a_j)\cap\Omega}$ 
for small $\rho>0$, 
so
\[
{-}\int_{\partial B_\rho(a_j)\cap\Omega} \psi \frac{\partial\psi}{\partial\rho} d\h^1 = \int_{\partial B_\rho(a_j)\cap\Omega} 
\left( R(z)-\sum_{k=1}^N d_k \log|z-a_k| \right) \left( \frac{d_j}{\rho}{ -}\frac{\partial R}{\partial\rho}- S_j(z)\right) \, d\h^1.
\]
Thus, since $\h^1(\partial B_\rho(a_j)\cap\Omega)=\pi\rho+o(\rho)$, we have for $\rho \to 0$: $\int_{\partial B_\rho(a_j)\cap\Omega}  R S_j \, d\h^1\to 0$, 
$$\frac{d_j}\rho\int_{\partial B_\rho(a_j)\cap\Omega}  R \, d\h^1\to \pi d_j R(a_j), \quad  \frac{d_j}\rho\int_{\partial B_\rho(a_j)\cap\Omega}  \log |z-a_k| \, d\h^1\to \pi d_j \log|a_j-a_k|, \, k\neq j,$$
\be
\label{aux23} \int_{\partial B_\rho(a_j)\cap\Omega}  R \frac{\partial R}{\partial\rho} \, d\h^1=\int_{B_\rho(a_j)\cap\Omega} |\nabla R|^2\, dx-\int_{B_\rho(a_j)\cap \partial \Omega}  R \frac{\partial R}{\partial\nu} \, d\h^1\to 0,
\ee
because $R\in H^1(\Om)\cap C^0(\bar\Om)$ and $\frac{\partial R}{\partial\nu}\in L^\infty(\dOm)$. Also, as in \eqref{aux23},  integration by parts in $B_\rho(a_j)\cap\Omega$ yields $\int_{\partial B_\rho(a_j)\cap\Omega}  \frac{\partial R}{\partial\rho} \zeta\, d\h^1\to 0$ as $\rho\to 0$ for the functions $\zeta=|\log \rho|$, resp. $\zeta(z)=\log|z-a_k|$ for $k\neq j$.
Summing after $j$, as $d_j^2=1$, the above estimates for $\rho\to 0$ lead to the representation formula \eqref{ren_1}
via Definition \ref{defi:renen}.
\qed
\end{proof}

We now present a somewhat more geometric argument to compute the renormalized energy in Definition \ref{defi:renen}.  This is based on the identification of the canonical harmonic maps with prescribed boundary vortices. 
In the following, we prove Theorem \ref{pro:canm}, i.e., an explicit formula of these canonical harmonic maps, first in the case of a disk domain, second on a general $C^{1,1}$ domain via a conformal Riemann map.\footnote{For a related calculation (with boundary values $\{\pm 1\}$ instead of $\pm$ the unit tangent), see Cabr\'e et al. \cite{CCK19}. }

\begin{proof}{ of Theorem \ref{pro:canm}}
The idea is to use the following two claims:

\medskip

\noindent {\bf Claim 1.} If $\Omega=\R^2_+$, then for any $N\geq 1$ distinct points $\{a_j\}_{1\leq j\leq N}$ on  $\partial \Om=\R\times \{0\}$ with multiplicities\footnote{In this case, there is {\bf no}  constraint on $\sum_{j=1}^N d_j$.}
$d_1, \dots, d_N\in\ZZ\setminus\{0\}$, the canonical harmonic map with prescribed boundary vortices $\{(a_j, d_j)\}$  has the form
$$m_*(z)=\pm \prod_{j=1}^N \left(\frac{z-a_j}{|z-a_j|} \right)^{d_j}, \quad \textrm{for all }\, z\in \R^2.$$

\medskip

\noindent {\bf Claim 2.} Let $\Phi:\overline{\omega}\to \overline{\Omega}$ be a $C^1$ conformal diffeomorphism with inverse $\Psi$ 
between two 
simply connected domains $\omega, \Omega\subset \R^2$. 
If $m_*$ is a canonical harmonic map with prescribed boundary vortices $\{(a_j, d_j)\}$ on $\partial \omega$ (where $\{a_j\}_{1\leq j\leq N}$ are distinct points on $\partial \omega$ and $d_j$ are non-zero integers satisfying $\sum_{j=1}^N d_j=2$), then $M_*$ given in \eqref{can_gen} 
is the canonical harmonic map with prescribed boundary vortices $\{(\Phi(a_j), d_j)\}$ on $\dOm$.

\medskip

\begin{proof}{ of Claim 1} In the case of the domain $\R^2_+$, a canonical harmonic map satisfies $m_*(z)=e^{\ci \phi_*(z)}$ in $z=(x,y)\in \R^2_+$ with $\Delta \phi_*=0$ in $\R^2_+$ and $\partial_{x} \phi_*=-\pi \sum_{j=1}^N d_j \delta_{a_j}$ on $\R\times \{0\}$ where $\phi_*(\R\times \{0\})\subset \pi \ZZ$. As in the proof of Proposition \ref{pro:ren_BBH}, any solution $\phi_*$ of this problem is determined via a conjugate harmonic function that has the form 
$\psi_*(z)=-\sum_{j=1}^N d_j \log |z-a_j|$ for $z\in \R^2_+$. So, if $\mathrm{Arg}$ is a
smooth argument in 
$\CC\setminus\R_+$, then $\phi_*(z)=\sum_{j=1}^N d_j \mathrm{Arg}(z-a_j)$ for $z\in \R^2_+$ is a solution of our problem, unique up to an additive constant in $\pi\ZZ$ .
\end{proof}

\medskip

\begin{proof}{ of Claim 2} As $\Phi$ is a conformal map, we know that
any point $z\in \partial \omega$ is mapped to $\Phi(z)\in\partial\Omega$; also, any unit tangent vector $v\in \Ss^1$ on $\partial \omega$ at $z$ is mapped to $\Phi'(z)v$, so the associated unit tangent on $\partial \Omega$ at $\Phi(z)$ is given by $\frac{v\Phi'(z)}{|\Phi'(z)|}$ and the orientation with respect to the outer normal fields at $\partial \omega$ and $\partial \Omega$ is preserved.  
This means that if $m_*$ is a canonical harmonic map with prescribed boundary vortices $\{(a_j, d_j)\}$ on $\partial \omega$, then the map given by $w=\Phi(z)\in \bar \Omega\mapsto 
\tilde m_*(z)=m_*(z)\frac{\Phi'(z)}{|\Phi'(z)|}$, in other words $M_*(w)=m_*(\Psi(w))\frac{\Phi'(\Psi(w))}{|\Phi'(\Psi(w))|}$ yields a map that is tangential to $\partial\Omega \setminus \{\Phi(a_j)\}$. Note that  $\Phi'(z)=|\Phi'(z)|e^{i \Theta(z)}$ for a smooth harmonic function $\Theta:\omega\to \R$
(recall that $\Phi'$ is holomorphic and never zero on the simply connected domain $\omega$, so it has a holomorphic logarithm).

 Since $m_*=e^{\ci \phi_*}$ in $\omega$, it yields $M_*=e^{\ci \tilde \phi_*}$ with 
\be
\label{conjugateee}
\tilde \phi_*(w)=\phi_*(\Psi(w))+\Theta(\Psi(w)) \quad \textrm{ for every } w\in \Omega
\ee yielding $\tilde \phi_*$ is harmonic in $\Omega$. 
 Since $\Theta$ is smooth, the degrees of $m_*$ near $a_j$ and $M_*$ near $\Phi(a_j)$ are the same.
\end{proof}

\medskip

Coming back to the proof of Theorem \ref{pro:canm}, in the case of the unit disk $\Omega=B_1$ and  $a_1, \dots, a_N\in \partial B_1$, we choose $\omega=\R^2_+$ and for each $b\in \partial B_1\setminus \{a_1, \dots , a_N\}$, we consider the conformal map $\Phi:\bar \R^2_+\to \bar B_1$ given by $\Phi(z)=b\frac{z-i}{z+i}$ for  $z\in \R^2_+\subset \CC$ with the inverse $\Psi(w)=i\frac{w+b}{b-w}$ for $w\in B_1$. Letting $\alpha_j=\Psi(a_j)$ for $j=1, \dots, N$ and $m_*$ be the canonical harmonic map in $\omega$ given at Claim 1 for prescribed boundary vortices $\{(\alpha_j, d_j)\}$ where $\sum_{j=1}^N d_j=2$, then one uses Claim 2 to deduce a canonical harmonic map $M_*$ as in \eqref{can_disk} with prescribed boundary vortices $\{(\alpha_j, d_j)\}$. 
(The constraint $\sum_{j=1}^N d_j=2$ is essential in the formula \eqref{can_disk} even if Claim 1 was independent of this contraint.) By the uniqueness of the canonical map (up to the transformation $M_*\mapsto -M_*$) for prescribed boundary vortices $\{(\alpha_j, d_j)\}$, we deduce the uniqueness in \eqref{can_disk}.\footnote{The uniqueness in \eqref{can_disk} yields the following fundamental identity: for every $b, \tilde b \in \partial B_1\setminus \{a_1, \dots, a_N\}$, then 
$$b \prod_{j=1}^N \left( \frac{|b-a_j|}{b-a_j} \right)^{d_j}=\pm \tilde b \prod_{j=1}^N \left( \frac{|\tilde b-a_j|}{\tilde b-a_j} \right)^{d_j}, 
$$ where $\sum_{j=1}^N d_j=2$. In fact, this identity can be proved as follows: if $b=e^{\ci \beta}$ and $a_j=e^{\ci \alpha_j}$, then $b-a_j=|b-a_j|e^{\ci \frac{\beta+\alpha_j\pm \pi}2}$ which yields the results due to the constraint $\sum_{j=1}^N d_j=2$ with $d_j\in \ZZ\setminus \{0\}$.
}
The general case of an arbitrary simply connected 
domain $\Omega$ follows by Claim 2. 
\qed
\end{proof}

\bigskip

We now prove the formula of the renormalized energy stated in Theorem \ref{pro:ree}:

\begin{proof}{ of Theorem \ref{pro:ree}}
The idea is to determine the solution $\psi$ in \eqref{eq:VisNeum}:

\medskip

\noindent {\bf Fact 1.} If $\Omega=\R^2_+$, then for any $N\geq 1$ distinct points $\{a_j\}_{1\leq j\leq N}$ on  $\partial \Om=\R\times \{0\}$ with multiplicities\footnote{In this case, there is {\bf no} constraint $\sum_{j=1}^N d_j=2$.} 
$d_1, \dots, d_N\in\ZZ\setminus\{0\}$, the solution (up to an additive constant) of  \eqref{eq:VisNeum} has the form 
$$\psi_*(z)=-\sum_{j=1}^N d_j \log |z-a_j| \quad \textrm{for all} \quad z\in \R^2.$$
(This is because $\frac{\partial}{\partial y} \log|z-a_j|=\pi \delta_{a_j}$ for $z=(x,y)\big|_{y=0}$.)

\medskip

\noindent {\bf Fact 2.} Let $\Phi:\overline{\omega}\to \overline{\Omega}$ be a $C^1$ conformal diffeomorphism with inverse $\Psi$ between two $C^{1,1}$ simply connected domains $\omega, \Omega\subset \R^2$. If $\psi_*$ is a solution of \eqref{eq:VisNeum} with prescribed boundary vortices $\{(a_j, d_j)\}$ on $\partial \omega$ (where $\{a_j\}_{1\leq j\leq N}$ are distinct points on $\partial \omega$ and $d_j$ are non-zero integers satisfying $\sum_{j=1}^N d_j=2$), then 
$$\tilde \psi_*(w)=\psi_*(\Psi(w)){+}\log |\Psi'(w)| \quad \textrm{ for every}\quad  w\in \Omega$$ 
is a solution of  \eqref{eq:VisNeum} with prescribed boundary vortices $\{(\Phi(a_j), d_j)\}$ on $\dOm$. (This is a consequence of \eqref{conjugateee} where $\tilde \phi_*(w)$ and $\phi_*(z)$ are harmonic conjugates\footnote{Recall our sign convention that $\phi_*+{\ci} \psi_*$ is holomorphic in $\omega$.} of $\tilde \psi_*(w)$ and $\psi_*(z)$ respectively, while $\Theta(z)$ is a conjugate harmonic of $-\log |\Phi'(z)|=\log |\Psi'(w)|$ for every $z=\Psi(w)\in \omega$.)

\medskip

\noindent {\bf Case 1. $\Omega=B_1$}. As in the proof of Theorem \ref{pro:canm}, we choose $\omega=\R^2_+$ and $b\in \partial B_1\setminus \{a_1, \dots , a_N\}$; then we consider the conformal map $\Phi:\bar \R^2_+\to \bar B_1$ given by $\Phi(z)=b\frac{z-i}{z+i}$ for every $z\in \R^2_+\subset \CC$ with inverse $\Psi(w)=i\frac{w+b}{b-w}$ for every $w\in B_1$. By Facts 1 and 2, a solution $\psi_*$ of \eqref{eq:VisNeum} in $B_1$ is given by
\begin{align}
\nonumber \psi_*(w)&=-\sum_j d_j \log|\Psi(w)-\Psi(a_j)|+\log \frac2{|b-w|^2}\\
\nonumber&=-\sum_j d_j \log\frac{2|w-a_j|}{|b-w|\cdot |b-a_j|}+\log \frac2{|b-w|^2}\\
\label{psi_b1}&=-\sum_j d_j \log |w-a_j|+R(w) \quad \textrm{for all } w\in B_1,
\end{align}
where $R$ is a constant function in $B_1$ (because of the constraint $\sum_j d_j=2$). Using that on $\partial B_1$, $\varkappa=1$ and\footnote{For $a\in \partial B_1$, $z\in B_1\mapsto \log|z-a|$ is harmonic, so the mean-value formula leads to $0=\log|z-a|\big|_{z=0}=\frac1{2\pi r}\int_{\partial B(0,r)} \log|z-a|\, d\h^1=\frac1{2\pi}\int_{\partial B_1} \log|rz-a|\, d\h^1\to\frac1{2\pi}\int_{\partial B_1} \log|z-a|\, d\h^1$ as $r\uparrow 1$ by dominated convergence theorem (due to the fact that $\log|z-a|\in L^1(\partial B_1)$).} $\int_{\partial B_1} \log|z-a| d\h^1=0$
  for all $a\in \partial B_1$,
we conclude to the desired formula via \eqref{ren_1} and Remark \ref{rem4}.
 
\medskip

\noindent {\bf Case 2. General domain $\Omega$}. We use Fact 2 and Case 1 where we can replace $\psi_*$ in \eqref{psi_b1} by $\psi_*-R$ as $R$ is a constant function (due to Remark \ref{rem4}). Therefore, the solution $\tilde \psi_*$ of \eqref{eq:VisNeum} in $\Omega$ is given by
$$\tilde \psi_*(w)=-\sum_j d_j \log|\Psi(w)-\Psi(a_j)|+\log |\Psi'(w)|\quad \textrm{for all } w\in \Om.$$
It follows that
$$
\tilde R(w)=\tilde \psi_*(w)+\sum_j d_j \log|w-a_j|=-\sum_j d_j \log\frac{|\Psi(w)-\Psi(a_j)|}{|w-a_j|}+\log |\Psi'(w)|.
$$
As $d_k^2=1$ for every $k$, it yields
$$\pi d_k \tilde R(a_k)=-\pi \sum_{j\neq k} d_k d_j \log\frac{|\Psi(a_k)-\Psi(a_j)|}{|a_k-a_j|}+\pi (d_k-1)\log |\Psi'(a_k)|$$
 and the desired formula follows by \eqref{ren_1}. \qed
\end{proof}

\begin{rem}
It is possible to also encode the effect of a small applied field or anisotropy in the renormalized energy, as has been done by Kurzke-Melcher-Moser \cite{KurzkeMelcherMoser:2013b}
 for interior vortices. In this case, the optimal phase 
 is no longer harmonic, but satisfies a nonlinear elliptic equation instead. 
\end{rem}

\begin{proof}{ of Corollary \ref{cor-ren-min}}
If $\Phi:\overline{B_1}\to \overline{\Omega}$ is a $C^1$ conformal diffeomorphism with inverse $\Psi$, setting $b_1=\Psi(a_1)\in \partial B_1$, $b_2=\Psi(a_2)\in \partial B_1$ for two distinct points  $a_1, a_2\in \dOm$, Remark \ref{rem_min_ellipse} yields
\begin{multline*}
W_\Omega(\{(a_1,1), (a_2,1)\})\\=
-2\pi \log |b_1-b_2| + \int_{\partial B_1} \varkappa( \Phi(z)) |\Phi'(z)| \bigg( \log |z-b_1|+\log|z-b_2| + \log |\Phi'(z)|\bigg) d\h^1(z).
\end{multline*}
Let $D=(\dOm\times \dOm)\setminus \{(a,a)\, :\, a\in \dOm\}$. 
In order to prove the existence of minimizers of $W_\Omega(\{(\cdot,1), (\cdot,1)\})$ over $D$, we consider a minimizing sequence 
$(a_1^{(n)}, a_2^{(n)})\in D$ for every $n\in \NN$. Set $b_1^{(n)}=\Psi(a_1^{(n)})$ and $b_2^{(n)}=\Psi(a_2^{(n)})$. As $\dOm$ is compact, up to a subsequence, we can assume that $a_1^{(n)}\to a_1^*$ and $a_2^{(n)}\to a_2^*$ as $n\to \infty$; thus, $b_1^{(n)}\to \Psi(a_1)=:b_1^*$ and $b_2^{(n)}\to \Psi(a_2)=:b_2^*$. Note that  
$$(b_1, b_2)\in \partial B_1\times \partial B_1\mapsto  \int_{\partial B_1} \varkappa( \Phi(z)) |\Phi'(z)| \bigg( \log |z-b_1|+\log|z-b_2| + \log |\Phi'(z)|\bigg) d\h^1(z)$$
is a bounded function as $\varkappa\in L^\infty(\dOm)$. As $\bigg(W_\Omega(\{(a_1^{(n)},1), (a_2^{(n)},1)\})\bigg)_n$ is bounded, it implies that $(\log |b_1^{(n)}-b_2^{(n)}|)_n$ is bounded; thus, $b_1^*\neq b_2^*$ and so, $a_1^*\neq a_2^*$ (as $\Psi$ is  injective), i.e., $(a_1^*, a_2^*)\in D$. By the continuity of $W_\Omega(\{(\cdot,1), (\cdot,1)\})$ over $D$, we deduce that  $(a_1^*, a_2^*)$ is a minimizer. If $\Omega=B_1$, then $W_\Omega(\{(a_1,1), (a_2,1)\})=
-2\pi \log |a_1-a_2|$ and any diameter $(a_1^*, a_2^*)$ minimizes the renormalized energy reaching  the minimal value $-2\pi \log 2$. \qed
\end{proof}

\appendix

\section{Existence and uniqueness of the stray field}
\label{appendixA}
We prove existence and uniqueness of the $H^1$ stray field potential $U$ in \eqref{Helmholtz} and we determine the exact formula of $U$.

\begin{pro}\label{pro:append}
Let $\bfom\subset \R^3$ be a bounded open set and $\bfm\in L^2(\bfom, \R^3)$. Then there exists a unique stray field potential $U\in H^1(\R^3)$ of the problem \eqref{Helmholtz}. The exact expression is given by the convolution 
$$U(\bfx)=-\frac{1}{4\pi |\bfx|}\star \bfa \cdot (\bfm \mathds{1}_{\bfom})\quad \textrm{in} \quad \R^3$$ 
of the distribution of compact support $\bfa \cdot (\bfm \mathds{1}_{\bfom})$ and the tempered distribution $V(\bfx)=-\frac{1}{4\pi |\bfx|}\in \mcS'(\R^3)$ that is the fundamental solution $\bfd V=\delta_0$ in $\R^3$. Moreover, if $\bfom$ is Lipschitz and $\bfm\in H^1(\bfom,\R^3)$, then
\be
\label{form_stray}
4\pi U(\bfx)=-\int_{\bfom} \frac{1}{|\bfx-\bfy|} \bfa\cdot  \bfm(\bfy) \, d\bfy+ \int_{\partial \bfom}\frac{1}{|\bfx-\bfy|} 
(\bfm\cdot \pmb{\nu})(\bfy)\, d\h^2(\bfy),\ee
where $\pmb{\nu}$ is the unit outer normal vector at $\partial \bfom$.
\end{pro}

\begin{proof}{}
  We apply Lax-Milgram's theorem for the problem \eqref{Helmholtz} in
  the Beppo-Levi space (in other words, the homogeneous $\dot{H}^1$-space): 
\begin{align*}
\mathcal{BL}&=\{U:\R^3\to \R\,:\, \bfa U\in L^2(\R^3), \frac{U}{1+|\bfx|}\in L^2(\R^3)\}\\
&=\{U\in \mcS'(\R^3)\,:\, \FF(U)\in L^1_{loc}(\R^3), |\bfxi|\FF(U)\in L^2(\R^3)\}{=: \dot H^1(\RR^3)},
\end{align*}
where $\FF(U)(\bfxi){=\int_{\RR^3} e^{-i\bfx\cdot \bfxi} U(\bfx)\, d\bfx}$ is the Fourier transform  of $U$ and $\bfxi$ is the Fourier variable in $\R^3$.
The space $\mathcal{BL}$ endowed with the homogeneous $\dot{H}^1$-norm, i.e., $U\mapsto \|\bfa U\|_{L^2(\R^3)}$, is a Hilbert space and the set $C^\infty_c(\R^3)$  of smooth compactly supported functions is a dense set. Since $\bfm\in  L^2$, Lax-Milgram's theorem yields the existence and uniqueness of the solution $U\in \mathcal{BL}$ of \eqref{Helmholtz}, in particular, $\bfa U\in L^2(\R^3)$, the Fourier transform $\FF(U)\in L^1_{loc}(\R^3)$ and we have
 $$
\bfd U = \bfa\cdot (\bfm \mathds{1}_{\bfom}) \quad \textrm{ in the sense of distributions in } \RR^3.
$$ This equation implies the following equality in the Fourier space:
$$\FF(U)(\bfxi)={-}\frac{i\bfxi\cdot \FF(\bfm \mathds{1}_{\bfom})(\bfxi)}{|\bfxi|^2}, \quad \bfxi\in \R^3\setminus \{0\}.$$
We check that $U\in L^2(\R^3)$. 
Indeed, 
\begin{align*}
 {(2\pi)^3} \|U\|^2_{L^{2}(\R^3)} \ &= \  \|\FF(U)\|^2_{L^{2}(\R^3)} \\
  &\leq \ \int_{|\bfxi|\geq 1} \frac{|\FF(\bfm \mathds{1}_{\bfom})(\bfxi)|^2}{|\bfxi|^2}\, d\bfxi +\int_{|\bfxi|\leq 1} \frac{|\FF(\bfm \mathds{1}_{\bfom})(\bfxi)|^2}{|\bfxi|^2}\, d\bfxi,  \\
  &\leq \ \nltL{\FF(\bfm \mathds{1}_{\bfom})}{\R^3}^2 + \|\FF(\bfm \mathds{1}_{\bfom})\|^2_{L^\infty(\R^3)} \int_{|\bfxi|\leq 1}  \frac 1{|\bfxi|^2} \ d\bfxi \\
  &\leq \ C \ \left (\| \bfm\|^2_{L^2(\bfom)} + \|\bfm\|^2_{L^1(\bfom)} \right) 
  \leq \ C \| \bfm\|^2_{L^2(\bfom)}.
\end{align*}
Let us check that the solution $U$ coincides with $\tilde{U}=V\star  \bfa \cdot (\bfm \mathds{1}_{\bfom})$ in $\R^3$. Indeed, we have that $V\in L^1(\R^3)+L^4(\R^3)$ (so, $V$ is a tempered distribution in 
$\mcS'(\R^3)$ {with $\FF(V)=-1/|\bfxi|^2$ for $\bfxi\neq 0$}) and $\bfa \cdot (\bfm \mathds{1}_{\bfom})$ is a distribution of compact support (because $\bfom$ is bounded); thus, $\tilde{U}$ is a tempered distribution in $\mcS'(\R^3)$ and we check that $|\bfxi| \FF(\tilde{U})=|\bfxi|\FF(V)\cdot \FF(\bfa \cdot (\bfm \mathds{1}_{\bfom}))\in L^2(\R^3)$, i.e.,
$$\int_{\R^3}|\bfxi|^2 |\FF(\tilde{U})|^2\, d\bfxi=\int_{\R^3}\frac{1}{|\bfxi|^2} |\FF(\bfa \cdot (\bfm \mathds{1}_{\bfom}))|^2\, d\bfxi
\leq \int_{\R^3} |\FF(\bfm \mathds{1}_{\bfom})|^2\, d\bfxi={(2\pi)^3}\| \bfm\|^2_{L^2(\bfom)}.$$ We conclude that $\tilde{U}$ belongs to 
$\dot H^1(\RR^3)$, satisfies $\bfd \tilde U = \bfa\cdot (\bfm \mathds{1}_{\bfom})$ in $\R^3$ and by the uniqueness of the solution $U$ of \eqref{Helmholtz}, it follows that $U=\tilde{U}$ in $\R^3$. In the case of a Lipschitz domain $\bfom$ and $\bfm\in H^1(\bfom,\R^3)$, one decomposes 
$$\bfa \cdot (\bfm \mathds{1}_{\bfom})=\bfa \cdot \bfm \h^3\llcorner {\bfom}- \bfm\cdot \pmb{\nu} \h^2\llcorner {\partial \bfom}$$
in the sense of measures in $\R^3$ and therefore, \eqref{form_stray} follows via $U=V\star  \bfa \cdot (\bfm \mathds{1}_{\bfom})$. \qed
\end{proof}

\bibliography{ignatkurzke}
\end{document}